\documentclass[A4paper,11pt]{amsart}

\usepackage[english]{babel}
\usepackage{amsmath}
\usepackage{amssymb} 
\usepackage{amsfonts}
\usepackage[all]{xy}
\usepackage{hyperref}
\usepackage{lscape}
 \newtheorem{proposition}{Proposition}[section] 
 \newtheorem{lemma}[proposition]{Lemma}
 \newtheorem*{lemma*}{Lemma}
 \newtheorem{theorem}[proposition]{Theorem}
 \newtheorem{corollary}[proposition]{Corollary}
\theoremstyle{definition}
 \newtheorem{definition}[proposition]{Definition}
 \newtheorem*{definition*}{Definition}
 \newtheorem{example}[proposition]{Example}
 \newtheorem*{example*}{Example}
 \newtheorem{remark}[proposition]{Remark}
 
\numberwithin{equation}{section}
\def\ox{\otimes}
\def\M{\mathbb M}
\def\id{\mathsf{id}}
\def\pibarL{\overline \sqcap^L}
\def\pibarR{\overline \sqcap^R}
\def\piL{\sqcap^L}
\def\piR{\sqcap^R}
\def\op{\mathsf{op}}
\def\wr{\raisebox{1pt}{\scalebox{.7}{$\,\triangleleft\,$}}}

\begin{document}

\title{Weak multiplier bialgebras}
\author{Gabriella B\"ohm}
\address{Wigner Research Centre for Physics, H-1525 Budapest 114,
P.O.B.\ 49, Hungary} 
\email{bohm.gabriella@wigner.mta.hu}
\author{Jos\'e G\'omez-Torrecillas}
\address{Departamento de \'Algebra, Universidad de Granada, 
E-18071 Granada, Spain} 
\email{gomezj@ugr.es}
\author{Esperanza L\'opez-Centella}
\address{Departamento de \'Algebra, Universidad de Granada,
E-18071 Granada, Spain}
\email{esperanza@ugr.es}
\thanks{Research partially supported by the Spanish Ministerio de Ciencia en
Innovaci\'on and the European Union, grant MTM2010-20940-C02-01; by
the Hungarian Scientific Research Fund OTKA, grant K108384; and by the
Nefim Fund of Wigner RCP. The authors thank Alfons Van Daele for valuable
discussions from which this paper benefits a lot.}
\begin{abstract}
A non-unital generalization of weak bialgebra is proposed with a
multi\-plier-valued comultiplication. Certain canonical subalgebras of the
multiplier algebra (named the `base algebras') are shown to carry coseparable
co-Frobenius coalgebra structures. Appropriate modules over a weak multiplier
bialgebra are shown to constitute a monoidal category via the (co)module
tensor product over the base (co)algebra. The relation to Van Daele and Wang's
(regular and arbitrary) weak multiplier Hopf algebra is discussed.
\end{abstract}
\date{Oct 2013}
\maketitle

\section*{Introduction}

The most well-known examples of {\em Hopf algebras} are the linear spans of
(arbitrary) {\em groups}. Dually, also the vector space of linear functionals
on a {\em finite} group carries the structure of a Hopf algebra. In the case
of {\em infinite} groups, however, the vector space of linear functionals ---
with finite support --- possesses no unit. Consequently, it is no longer a
Hopf algebra but, more generally, a {\em multiplier Hopf algebra}
\cite{VDae:MHA}. Replacing groups with {\em finite groupoids}, both their
linear spans and the dual vector spaces of linear functionals carry {\em weak
Hopf algebra} structures \cite{WHAI}. Finally, removing the finiteness
constraint in this situation, both the linear spans of arbitrary groupoids,
and the vector spaces of linear functionals with finite support on them are
examples of {\em weak multiplier Hopf algebras} as introduced in the recent
paper \cite{VDaWa}. 

Van Daele's approach to multiplier Hopf algebras is based on the principle of
using minimal input data. That is, one starts with a non-unital algebra $A$
with an appropriately well-behaving multiplication and a multiplicative map
$\Delta$ from $A$ to the multiplier algebra of $A\ox A$. This allows one to
define maps $T_1$ and $T_2$ from $A\ox A$ to the multiplier algebra of $A\ox
A$ as 
$$
T_1(a\ox b)=\Delta(a)(1\ox b) \qquad \textrm{and}\qquad
T_2(a\ox b)=(a\ox 1)\Delta(b),
$$
where $1$ stands for the unit of the multiplier algebra of $A$.
(If $A$ is a usual, unital bialgebra over a field $k$, then these maps are
the left and right Galois maps for the $A$-extension $k\to A$ provided by the
unit of $A$.) The axioms of multiplier Hopf algebra assert first that $T_1$
and $T_2$ establish isomorphisms from $A\ox A$, to $A\ox A$ regarded as a
two-sided ideal in the multiplier algebra. Second, $T_1$ and $T_2$ are
required to obey $(T_2\ox \id)(\id \ox T_1)=(\id \ox T_1) (T_2\ox \id)$
(replacing the coassociativity of $\Delta$ in the unital case). These axioms
are in turn equivalent to the existence of a counit and an antipode with the
expected properties. In particular, if $A$ has a unit, then it is a multiplier
Hopf algebra if and only if it is a Hopf algebra.

A similar philosophy is applied in \cite{VDaWa} to define weak multiplier
Hopf algebra. Recall that if $A$ is a weak Hopf algebra over a field $k$ with
a unit $1$, then its comultiplication $\Delta$ is not required to preserve $1$
(i.e. $\Delta(1)$ may differ from $1\ox 1$). Consequently, the maps $T_1$ and 
$T_2$ are no longer linear automorphisms of $A\ox A$. Instead, they induce
isomorphisms between some canonical vector subspaces determined by the element
$\Delta(1)$. In the situation when $A$ is allowed to possess no unit, in
\cite{VDaWa} the role of $\Delta(1)$ is played by an idempotent element in
the multiplier algebra of $A\ox A$, which is meant to be part of the
structure. It is used to single out some canonical vector subspaces of $A\ox
A$. The maps $T_1$ and $T_2$ are required to induce isomorphisms between these
vector subspaces and the same (coassociativity) axiom $(T_2\ox \id)(\id \ox
T_1)=(\id \ox T_1) (T_2\ox \id)$ is imposed. In contrast to the case of
non-weak multiplier Hopf algebras, however, these axioms do not seem to imply
the existence and the expected properties of the counit and the
antipode. Therefore in \cite{VDaWa} also the existence of a counit $\epsilon:A
\to k$ is assumed (in the sense that $(\epsilon\ox \id)T_1$ and $(\id\ox
\epsilon)T_2$ are equal to the multiplication on $A$). Adding these counit
axioms, the existence of the antipode and {\em most} of the expected
properties of the counit and the antipode do follow. However --- at least
without requiring that the opposite algebra obeys the same set of axioms,
called the {\em regularity} condition in \cite{VDaWa} --- some crucial
properties seem to be missing (see \cite{VDaWa} for several discussions on
this issue). Most significantly, in a usual, unital weak Hopf algebra, the
counit $\epsilon$ is required to obey two symmetrical conditions
\begin{equation}\label{eq:wmulti}\tag{wm}
(\epsilon\ox\epsilon)((a\ox 1)\Delta(b)(1\ox c))=\epsilon(abc)=
(\epsilon\ox\epsilon)((a\ox 1)\Delta^{\op}(b)(1\ox c)),\qquad \forall a,b,c\in
A. 
\end{equation}
Interestingly enough, the axioms of weak multiplier Hopf algebra in
\cite{VDaWa} imply the second equality but apparently not the first one
(unless regularity is assumed). In this way, even if a weak multiplier Hopf
algebra has a unit, it may not be a usual, unital weak Hopf algebra.

One aim of this paper is to identify an intermediate class between regular and
arbitrary weak multiplier Hopf algebras in \cite{VDaWa}. This class should be
big enough to contain any usual weak Hopf algebra. On the other hand, its
members should have the expected properties like the (separable Frobenius
type) structure of the base algebras.

Achieving the above aim, we depart from a bit further. If considering monoids
instead of groups, their linear spans (and vector spaces of functionals in the
finite case) are only {\em bialgebras}, no longer Hopf algebras. Similarly,
the linear spans of small categories with finitely many objects (and the
vector spaces of functionals in the case when also the number of arrows is
finite) are only {\em weak bialgebras} but not weak Hopf algebras. So with the
ultimate aim to describe the analogous structures associated to categories
without any finiteness assumption, we study {\em weak multiplier
bialgebras}. In this case the existence and the appropriate properties of the
counit have to be assumed. In Section \ref{sec:ax} we propose a set of axioms
defining a weak multiplier bialgebra and we present it in several equivalent
forms. We show that any {\em regular} weak multiplier Hopf algebra
obeys these axioms and so does any weak bialgebra (with a unit). By
generalizing to the multiplier setting several equivalent properties that
distinguish bialgebras among weak bialgebras, we also propose a notion of {\em
multiplier bialgebra} (which is, however, different from both notions in
\cite{JaVe} and \cite{Tim} occurring under the same name). In Section
\ref{sec:base} and Section \ref{sec:sF} we study some distinguished
subalgebras of the multiplier algebra of a weak multiplier bialgebra. They
generalize the `source' and `target' (also called `right' and `left') base
algebras of a unital weak bialgebra. Whenever the comultiplication is `full'
(in the sense of \cite{VDaWa}), they are shown to carry firm Frobenius algebra
structures arising from a coseparable co-Frobenius coalgebra in the sense of
\cite{BoGT:fF}. In Section \ref{sec:modcat} we study an appropriate category
of modules over a regular weak multiplier bialgebra with a full
comultiplication. It is shown to be a monoidal category equipped with a strict
monoidal and faithful (in some sense `forgetful') functor to the category of
firm bimodules over the base algebra. In Section \ref{sec:antipode} we
introduce the notion of antipode on a regular weak multiplier
bialgebra. Whenever the comultiplication is full, the antipode axioms are
shown to be equivalent to the projections of the maps $T_1$ and $T_2$, to maps
between relative tensor products over the base algebras, being
isomorphisms. The resulting structure is equivalent to a weak multiplier Hopf
algebra in the sense of \cite{VDaWa} such that in addition both conditions in
\eqref{eq:wmulti} hold. We claim that this is the desired `intermediate' class
between regular and arbitrary weak multiplier Hopf algebras in which one can
answer the questions left open in \cite{VDaWa} and which is big enough to
contain any unital weak Hopf algebra. Some preliminary information on
multiplier algebras are collected in Section \ref{sec:prelims}. 

\section{Preliminaries on multiplier algebras}\label{sec:prelims}

For a field $k$, we denote by $\ox$ the tensor product of $k$-vector spaces. 
Let $A$ be a vector space over $k$. If there is an associative multiplication
$$
\mu: A\ox A\to A, \qquad a\ox b \mapsto ab
$$
then we say that $A$ is a {\em non-unital} algebra. If in addition $\mu$ is
surjective then $A$ is said to be an {\em idempotent} algebra. Furthermore, if
any of the conditions $(ab=0,\ \forall a\in A)$ and $(ba=0,\ \forall a\in A)$
implies $b=0$, then the multiplication $\mu$ is termed {\em non-degenerate}.
Clearly, also the algebra $A^\op$ on the same vector space $A$ with the
opposite multiplication is idempotent and non-degenerate, whenever A is
so. If $A$ and $B$ are idempotent and non-degenerate algebras, then so
is $A\ox B$ with the factorwise multiplication, see e.g. \cite[Lemma
1.11]{JaVe}. 

Let $A$ be a non-unital algebra with a non-degenerate multiplication. A {\em
multiplier} on $A$ \cite{Dauns:Multiplier} is a pair $(\lambda,\rho)$ of
$k$-linear maps $A\to A$ such that $a\lambda(b)=\rho(a)b$ for all $a,b\in A$. 
Then it follows that $\lambda$ is a morphism of right $A$-modules and
$\rho$ is a map of left $A$-modules. The vector space of multipliers on $A$
--- via the componentwise linear structure --- is denoted by $\M(A)$. It is an 
associative algebra via the multiplication
$(\lambda',\rho')(\lambda,\rho)=(\lambda'\lambda,\rho\rho')$ (where 
juxtaposition means composition) and unit $1=(\id,\id)$. Any element $a$ of
$A$ can be regarded as a multiplier as $(b\mapsto ab,b\mapsto ba)$. This
allows us to regard $A$ as a dense two-sided ideal in $\M(A)$: Indeed, for
$(\lambda,\rho)\in \M(A)$ and $a\in A$, $a(\lambda,\rho)= \rho(a)$ and
$(\lambda,\rho)a=\lambda(a)$; and --- by non-degeneracy of the multiplication
--- $\rho=0$ if and only if $\lambda=0$. Clearly, $\M(A)^\op\cong \M(A^\op)$.
If $B$ denotes a second non-unital algebra with a non-degenerate
multiplication, then we have algebra embeddings $A\ox B \subseteq \M(A) \ox
\M(B) \subseteq \M(A\ox B)$. None of these inclusions will be explicitly
denoted throughout the paper. The multiplication in $\M(A)$ will also be
denoted by $\mu:\M(A)\ox \M(A)\to \M(A)$. 

Along this paper, if $X$ is a set of vectors of some vector space, we will
denote by $\langle X \rangle$ the vector subspace linearly spanned
 by $X$.

\begin{theorem}\cite[Proposition A.3]{VDaWa:Banach}\label{thm:extend}
Let $A$ and $B$ be non-unital algebras with non-degenerate multiplications and
$\gamma : A \to \M (B)$ be a multiplicative linear map. Assume that there is
an idempotent element $e \in \M (B)$ such that 
$$
\langle \gamma(a)b\ |\ a\in A\ b\in B \rangle = \{eb\ |\ b\in B \}
\quad \textrm{and}\quad 
\langle b\gamma(a)\ |\ a\in A\ b\in B \rangle = \{ be\ |\ b\in B \}. 
$$
Then there is a unique multiplicative linear map $\overline
\gamma : \M (A) \to \M (B)$ such that $\overline\gamma(1) =
e$ and $\overline \gamma(a)=\gamma (a)$, for all $a\in A$. 
\end{theorem}

If for some map $\gamma$ there exists an idempotent element $e$ as in Theorem
\ref{thm:extend}, then it is clearly unique (cf. \cite[Proposition
1.6]{VDaWa}). 

Let $A$ be a non-unital idempotent algebra with a non-degenerate
multiplication and let $\Delta:A\to \M(A\ox A)$ be a multiplicative linear
map. Assume that there is an idempotent element $E \in \M (A\ox A)$ such that 
\begin{eqnarray*}
&& \langle \Delta(a)(b\ox b')\ |\ a,b,b'\in A \rangle =
\langle E(b\ox b')\ |\ b,b'\in A \rangle 
\quad \textrm{and}\\
&& \langle (b\ox b')\Delta(a)\ |\ a,b,b'\in A \rangle =
\langle (b\ox b')E\ |\ b,b'\in A \rangle 
\end{eqnarray*}
as vector spaces. Then by Theorem \ref{thm:extend}, there exist the 
extended multiplicative maps $\overline \Delta:\M(A)\to \M(A\ox A)$,
$\overline {\Delta \ox \id}: \M(A\ox A)\to \M(A\ox A\ox A)$ and $\overline
{\id \ox \Delta}:\M(A\ox A)\to \M(A\ox A\ox A)$. If the ranges of the
maps
$$
T_1(a\ox b)=\Delta(a)(1\ox b) \quad \textrm{and}\qquad
T_2(a\ox b)=(a\ox 1)\Delta(b),\quad \textrm{for}\ a\ox b \in A\ox A,
$$
belong to $A\ox A$ and they satisfy $(T_2\ox \id)(\id \ox
T_1)=(\id \ox T_1) (T_2\ox \id)$, then it follows by
\cite[Proposition A.6]{VDaWa:Banach} that
$(\overline {\id \ox \Delta})(E)=
(\overline {\Delta \ox \id})(E)$. This allows us to define the
idempotent element
\begin{equation}\label{eq:E_3}
E^{(3)}:=(\overline {\id \ox \Delta})(E)=
(\overline {\Delta \ox \id})(E)
\end{equation} 
in $\M(A\ox A\ox A)$.

For any $k$-vector space $A$, denote by $\mathsf{tw}$ the flip map $A\ox A \to
A\ox A$, $a\ox b\mapsto b\ox a$.
For a non-unital algebra $A$ with a non-degenerate multiplication, and for a
multiplicative linear map $\Delta:A\to \M(A \ox A)$, define a multiplicative
linear map $\Delta^{\op}:A\to
\M(A\ox A)$ via 
$$
\Delta^{\op}(a)(b\ox c):=\mathsf{tw}(\Delta(a)(c \ox b))
\quad \textrm{and}\quad
(b\ox c)\Delta^{\op}(a):=\mathsf{tw}((c \ox b)\Delta(a))
$$ 
and define $\Delta_{13}:A\to \M(A\ox A\ox A)$ by 
\begin{eqnarray*}
&&\Delta_{13}(a)(b\ox c\ox d):=(\id \ox \mathsf{tw})(\Delta(a)(b\ox d)\ox c)
\quad \textrm{and}\\
&&(b\ox c\ox d)\Delta_{13}(a):=(\id \ox \mathsf{tw})((b\ox d)\Delta(a)\ox c).
\end{eqnarray*}

\section{The weak multiplier bialgebra axioms}\label{sec:ax}

The central notion of the paper, weak multiplier bialgebra, is introduced in
this section. Several equivalent forms of the axioms are presented and their
first consequences are drawn. Some illustrative examples are collected.

\begin{definition}\label{def:mwba}
A {\em weak multiplier bialgebra} $A$ over a field $k$ is given by
\begin{itemize}
\item an idempotent $k$-algebra with a non-degenerate multiplication $\mu:A\ox
A \to A$,
\item an idempotent element $E$ in $\M(A\ox A)$,
\item a multiplicative linear map $\Delta:A \to \M(A\ox A)$ (called the {\em
comultiplication}), 
\item and a linear map $\epsilon:A\to k$ (called the {\em counit}),
\end{itemize}
which are subject to the axioms below. 
\begin{itemize}
\item[{(i)}] For any elements $a,b\in A$, the elements 
$$
T_1(a\ox b):=\Delta(a)(1\ox b), \qquad \textrm{and} \qquad
T_2(a\ox b):=(a\ox 1)\Delta(b)
$$
of $\M(A\ox A)$ belong to the two-sided ideal $A\ox A$.
\item[{(ii)}] The comultiplication is coassociative in the sense that 
$$
(T_2\ox \id)(\id \ox T_1)=
(\id \ox T_1) (T_2\ox \id).
$$
\item[{(iii)}] The counit obeys 
$$
(\epsilon\ox \id)T_1=\mu=(\id\ox \epsilon)T_2.
$$
\item[{(iv)}] In terms of the idempotent element $E$, 
\begin{eqnarray*}
&& \langle \Delta(a)(b\ox b')\ |\ a,b,b'\in A \rangle =
\langle E(b\ox b')\ |\ b,b'\in A \rangle 
\quad \textrm{and}\\
&&\langle (b\ox b')\Delta(a)\ |\ a,b,b'\in A \rangle =
\langle (b\ox b')E\ |\ b,b'\in A \rangle. 
\end{eqnarray*}
\item[{(v)}] The idempotent element $E$ satisfies the equality
$$
(E\ox 1)(1\ox E)=E^{(3)}=(1\ox E)(E\ox 1) 
$$
in $\M(A\ox A\ox A)$, cf. \eqref{eq:E_3}.
\item[{(vi)}] For any $a,b,c\in A$,
\begin{eqnarray*}
&&(\epsilon\ox \id)((1\ox a)E(b\ox c))=(\epsilon\ox \id)(\Delta(a)(b\ox c))
\quad \textrm{and}\\
&&(\epsilon\ox \id)((a\ox b)E(1\ox c))=(\epsilon\ox \id)((a\ox b)\Delta(c)).
\end{eqnarray*}
\end{itemize}
\end{definition}

It follows immediately from axiom (iv) and the idempotency of $E$ that
$E\Delta(a)=\Delta(a)=\Delta(a) E$, for all $a\in A$. 

\begin{remark}\label{rem:E_unique}
In a weak multiplier bialgebra, the idempotent element $E$ and the counit
$\epsilon$ are uniquely determined in fact by the multiplication $\mu$ and the
comultiplication $\Delta$. The uniqueness of $E$ follows by the uniqueness of
the idempotent element in Theorem \ref{thm:extend}. We will come back to the
uniqueness of $\epsilon$ later in this section (cf. Theorem \ref{thm:unique}). 
\end{remark}

\begin{definition}\label{def:regular}
A weak multiplier bialgebra $A$ is said to be {\em regular} if also the
elements
$$
T_3(a\ox b):=(1\ox b)\Delta(a) \quad \textrm{and}\quad 
T_4(a\ox b):=\Delta(b)(a\ox 1)
$$
of $\M(A\ox A)$ belong to the two-sided ideal $A\ox A$, for all $a,b \in A$.
\end{definition}

For a regular weak multiplier bialgebra $A$, some of the axioms can be
re-written in the following equivalent forms.
\begin{itemize}
\item[{(ii)}] $\Leftrightarrow (T_4\ox \id)(\id \ox T_3)= (\id \ox T_3)
(T_4\ox \id)$, 
\item[{(iii)}] $\Leftrightarrow (\epsilon\ox \id)T_3=\mu^\op=(\id\ox
\epsilon)T_4$. 
\end{itemize}
So a weak multiplier bialgebra $A$ over a field is regular if and only if the
opposite algebra $A^\op$ is a weak multiplier bialgebra too, via the same
comultiplication $\Delta$, counit $\epsilon$ and idempotent element $E$. 

Below we shall provide some equivalent forms of axiom (vi) in Definition
\ref{def:mwba}. In particular, this will allow us to prove the uniqueness of
the counit.

\begin{proposition}\label{prop:pibarL}
For any weak multiplier bialgebra $A$ over a field, and for any $a\in A$, the 
linear maps $A\to A$,
\begin{equation}\label{eq:pibarL}
b\mapsto (\epsilon \ox \id)T_2(a\ox b)\quad \textrm{and}\quad
b\mapsto (\epsilon \ox \id)((a\ox b)E)
\end{equation}
define a multiplier $\pibarL(a)$ on $A$, giving rise to a linear map
$\pibarL:A \to \M(A)$. 
\end{proposition}

\begin{proof}
For any $a,b,c\in A$,
\begin{eqnarray*}
c((\epsilon \ox \id)T_2(a\ox b))&=&
(\epsilon \ox \id)((a\ox c)\Delta(b))\\
&\stackrel{(vi)}=&
(\epsilon \ox \id)((a\ox c)E(1\ox b))=
((\epsilon \ox \id)((a\ox c)E))b.
\end{eqnarray*}
\end{proof}

\begin{proposition}\label{prop:(a@1)E}
Let $A$ be a weak multiplier bialgebra over a field. For any $a,b\in A$,
the following assertions hold.
\begin{itemize}
\item[{(1)}] $(\id \ox \pibarL)T_2(a\ox b)=(ab\ox 1)E$
as elements of $\M(A\ox A)$. 
\item[{(2)}] $(a\ox 1)E$ belongs to the (non-unital) subalgebra $A\ox \M(A)$
of $\M(A\ox A)$.
\item[{(3)}] $(a\ox 1)E(1\ox b)$ belongs to the (non-unital) subalgebra $A\ox
A$ of $\M(A\ox A)$.
\end{itemize}
\end{proposition}

\begin{proof} (1). For any $a,b,p,q\in A$, 
\begin{eqnarray*}
(p\ox q)((\id \ox \pibarL)T_2(a\ox b)) &=&
(\id \ox \epsilon \ox \id)[(T_2(pa\ox b)\ox q)(1\ox E)]\\
&\stackrel{(iv)}=&
(\id \ox \epsilon \ox \id)[(T_2(pa\ox b)\ox q)(E\ox 1)(1\ox E)]\\
&\stackrel{(v)}=&
(\id \ox \epsilon \ox \id)[(T_2 \ox \id)(pa\ox b\ox q)E^{(3)}]\\
&=&(\id \ox \epsilon \ox \id)(T_2 \ox \id)[(pa\ox b\ox q)(1\ox E)]\\
&\stackrel{(iii)}=&
(p\ox q)(ab\ox 1)E.
\end{eqnarray*}
In the first equality we used the definition of $\pibarL$ in Proposition
\ref{prop:pibarL} and the left $A$-module map property of $T_2$. 
The fourth equality follows by
\begin{eqnarray*}
((T_2 \ox \id)(a\ox b\ox c))E^{(3)}&\stackrel{\eqref{eq:E_3}}=&
(a\ox 1\ox 1)(\Delta \ox \id)(b\ox c)(\overline{\Delta \ox \id})(E)\\
&=&(a\ox 1\ox 1)((\Delta \ox \id)((b\ox c)E))\\
&=&(T_2 \ox \id)((a\ox b\ox c)(1\ox E)),
\end{eqnarray*}
for any $a,b,c\in A$. 

(2). In the equality in (1), the left hand side belongs to $A\ox \M(A)$ hence
so does the right hand side. Since $A$ is an idempotent algebra by assumption,
this proves (2).

(3) follows immediately from (2), since $A$ is an ideal of $\M(A)$. 
\end{proof}

\begin{proposition}\label{prop:vi_versions}
Let $A$ be an idempotent algebra over a field $k$ with a non-degenerate
multiplication, $\Delta:A\to \M(A\ox A)$ be a multiplicative linear map,
$\epsilon:A \to k$ be a linear map and $E$ be an idempotent element in
$\M(A\ox A)$. Assume that the axioms (i)-(v) --- but not necessarily (vi) ---
in Definition \ref{def:mwba} hold. The following assertions are
equivalent \footnote{The proof of $(4)\Rightarrow (1)$ was kindly communicated
to us by A. Van Daele.}. 
\begin{itemize}
\item[{(1)}] $(\epsilon \ox \id)((a\ox b)E(1\ox c))=(\epsilon \ox \id)((a\ox
b) \Delta(c))$ for all $a,b,c\in A$.
\item[{(2)}] $(\id \ox \epsilon)((a\ox 1)E(b\ox c))=(\id \ox \epsilon)
(\Delta(a)(b\ox c))$ for all $a,b,c\in A$. 
\item[{(3)}] $(a\ox 1)E(1\ox c)\in A\ox A$ and $(\epsilon \ox \epsilon)((a\ox
1)E(1\ox c))=\epsilon(ac)$ for all $a,c\in A$.
\item[{(4)}] $(\epsilon \ox \epsilon)((a\ox 1)\Delta(b)(1\ox
c))=\epsilon(abc)$ for all $a,b,c\in A$. 
\end{itemize}
\end{proposition}

\begin{proof}
(1)$\Rightarrow$(3). Note that (1) is in fact the second one of the axioms in
Definition \ref{def:mwba}~(vi). Hence the same reasoning used to prove
Proposition \ref{prop:(a@1)E}~(3) shows that for all $a,b\in A$, $(a\ox
1)E(1\ox b)\in A\ox A$; so that (1) is equivalent to $(\epsilon \ox \id)((a\ox
1)E(1\ox c))=(\epsilon \ox \id)T_2(a\ox c)$ for all $a,c\in A$. Applying
$\epsilon$ to both sides of this equality and using the counitality axiom
(iii), we obtain the equality in (3). 

(2)$\Rightarrow$(3). Symmetrically, if (2) holds then for any $a\in A$ the
maps
\begin{equation}\label{eq:pibarR}
b\mapsto (\id \ox \epsilon)(E(b\ox a))\quad \textrm{and}\quad
b\mapsto (\id \ox \epsilon)T_1(b\ox a)
\end{equation}
define a multiplier $\pibarR(a)$ on $A$, for which 
\begin{equation}\label{eq:E(1@a)}
(\pibarR\ox \id)T_1(a\ox b)=E(1\ox ab), \qquad \forall a,b\in A. 
\end{equation}
This proves $E(1\ox b)\in \M(A)\ox A$ hence
$(a\ox 1)E(1\ox b)\in A\ox A$. Then (2) is equivalent to 
\begin{equation}\label{equiv(2)}
(\id \ox
\epsilon)((a\ox 1)E(1\ox c))=(\id \ox \epsilon) T_1(a \ox c)
\end{equation} 
for all $a,c\in A$. Applying $\epsilon$ to both sides of this equality and
using the counitality axiom (iii), we obtain the equality in (3). 

(3)$\Rightarrow$(1). For any $a,b,c\in A$, 
\begin{eqnarray*}
(\epsilon \ox \id)((a\ox 1)\!\!\!\!\!&\!\!\!\!\!E\!\!\!\!\!&\!\!\!\!\!
(1\ox b))c \stackrel{(iii)}=
(\epsilon\ox \epsilon\ox \id)(\id \ox T_1)((a\ox 1\ox 1)(E\ox 1)
(1\ox b\ox c))\\ 
&\stackrel{(v)}=&
(\epsilon\ox \epsilon\ox \id)((a\ox 1\ox 1)(E\ox 1)(1\ox E)
(1\ox T_1(b\ox c)))\\
&\stackrel{(iv)}=&
(\epsilon\ox \epsilon\ox \id)((a\ox 1\ox 1)(E\ox 1)(1\ox T_1(b\ox c)))\\
&\stackrel{(3)}=&
(\epsilon\ox \id)((a\ox 1)T_1(b\ox c))=
((\epsilon\ox \id)T_2(a\ox b))c,
\end{eqnarray*}
so we conclude by the non-degeneracy of the multiplication in $A$.

(3)$\Rightarrow$(2) is proven symmetrically.

(1) (and (3)) $\Rightarrow$(4). For any $a,b,c\in A$, 
$$
(\epsilon\ox \epsilon)((a\ox 1)\Delta(b)(1\ox c))\stackrel{(1)}=
(\epsilon\ox \epsilon)((a\ox 1)E(1\ox bc))\stackrel{(3)}=
\epsilon(abc).
$$

(4)$\Rightarrow$(1). For the idea of the reasoning below, we are grateful to
Fons Van Daele. 

\noindent
In view of axiom (iv) in Definition \ref{def:mwba}, (1) is equivalent to 
$$
(\epsilon \ox \id)((a\ox b)\Delta(c)(1\ox d))=(\epsilon \ox \id)((a\ox
b) \Delta(cd))\ \forall a,b,c,d\in A,
$$
hence by the non-degeneracy of the multiplication, also to
$$
(\epsilon \ox \id)T_2(a\ox c)d=(\epsilon \ox \id)T_2(a\ox cd)\ 
\forall a,c,d\in A.
$$
So we will prove it in this last form. For any $c,d\in A$, denote $c'\ox
d':=T_1(c\ox d)$ (allowing for implicit summation). Then for any $b\in A$,
\begin{equation}\label{eq:T_1_1st_arg}
T_1(bc\ox d)=\Delta(b)(c'\ox d')=T_1(b\ox d')(c'\ox 1).
\end{equation}
With this information in mind, for any $a,b,c,d\in A$, 
\begin{eqnarray*}
((\epsilon \ox \id)T_2(a\ox b))cd&=&
(\epsilon \ox \id)[T_2(a\ox b)(1\ox cd)]\\
&\stackrel{(iii)}=&
(\epsilon \ox \epsilon \ox \id)(\id \ox T_1)[(T_2(a\ox b)(1\ox c))\ox d]\\
&\stackrel{\eqref{eq:T_1_1st_arg}}=&
(\epsilon \ox \epsilon \ox \id)[((\id \ox T_1)(T_2\ox \id)(a\ox b \ox d'))
(1\ox c'\ox 1)]\\
&\stackrel{(ii)}=&
(\epsilon \ox \epsilon \ox \id)[((T_2\ox \id)(\id \ox T_1)(a\ox b \ox d'))
(1\ox c'\ox 1)]\\
&\stackrel{(4)}=&
(\epsilon \ox \id)[(a\ox 1)T_1(b\ox d')(c'\ox 1)]\\
&\stackrel{\eqref{eq:T_1_1st_arg}}=&
(\epsilon \ox \id)[(a\ox 1)T_1(bc\ox d)]=
((\epsilon \ox \id)T_2(a\ox bc))d,
\end{eqnarray*}
so we conclude by the non-degeneracy of the multiplication. 
\end{proof}

The following symmetric version is immediate.

\begin{proposition}\label{prop:vi_regular}
Let $A$ be an idempotent algebra over a field $k$ with a non-degenerate
multiplication, $\Delta:A\to \M(A\ox A)$ be a multiplicative linear map,
$\epsilon:A \to k$ be a linear map and $E$ be an idempotent element in
$\M(A\ox A)$. Assume that also the ranges of the maps $T_3$ and $T_4$ in
Definition \ref{def:regular} are in the ideal $A\ox A$, and that the axioms
(i)-(v) --- but not necessarily (vi) --- in Definition \ref{def:mwba}
hold. The following assertions are equivalent.
\begin{itemize}
\item[{(1)}] $(\epsilon \ox \id)((1\ox a)E(b\ox c))=(\epsilon \ox
\id)(\Delta(a)(b \ox c)$, for any $a,b,c\in A$. 
\item[{(2)}] $(\id \ox \epsilon)((a\ox b)E(c\ox 1))=(\id \ox \epsilon)
((a\ox b)\Delta(c))$, for any $a,b,c\in A$. 
\item[{(3)}] $(1\ox a)E(c\ox 1)\in A\ox A$ and $(\epsilon \ox \epsilon)((1\ox
a)E(c\ox 1))=\epsilon(ac)$, for any $a,c\in A$. 
\item[{(4)}] $(\epsilon \ox \epsilon)((1\ox a)\Delta(b)(c\ox
1))=\epsilon(abc)$, for any $a,b,c\in A$. 
\end{itemize}
\end{proposition}

\begin{theorem}\label{thm:unique}
The counit of a weak multiplier bialgebra $A$ over a field $k$ is uniquely
determined by the multiplication and the comultiplication. 
\end{theorem}

\begin{proof}
We have seen in Remark \ref{rem:E_unique} that the idempotent element $E$ is
uniquely fixed. Let $\epsilon,\epsilon':A\to k$ be counits for $A$. Then for
all $a,b,c\in A$,
\begin{eqnarray*}
(\epsilon\ox \epsilon')[(a\ox 1)\!\!\!&\!\!\!\Delta(b)\!\!\!&\!\!\!(1\ox c)]=
(\epsilon\ox \epsilon')[(a\ox 1)T_1(b\ox c)]\\
&=&
(\epsilon\ox \epsilon\ox \epsilon')[(a\ox 1\ox 1)(E\ox 1)(1\ox T_1(b\ox c))]\\
&\stackrel{(iv)}=&
(\epsilon\ox \epsilon\ox \epsilon')[(a\ox 1\ox 1)(E\ox 1)(1\ox E)
(1\ox T_1(b\ox c))]\\
&\stackrel{(v)}=&
(\epsilon\ox \epsilon\ox \epsilon')[(a\ox 1\ox 1)(\overline{\id\ox
\Delta})(E(1\ox b))(1\ox 1\ox c)]\\
&=&(\epsilon\ox \epsilon\ox \epsilon')[(a\ox 1\ox 1)
(\id \ox T_1)(E(1\ox b) \ox c)]\\
&\stackrel{(iii)}=&
(\epsilon\ox \epsilon')[(a\ox 1)E(1\ox bc)]
=(\epsilon\ox \epsilon') T_1(a\ox bc) \stackrel{(iii)}
=\epsilon'(abc).
\end{eqnarray*}
In the second and the penultimate equalities we used Proposition
\ref{prop:vi_versions}~(3) and (2) (in its alternative form \eqref{equiv(2)})
for $\epsilon$ and for $\epsilon'$, respectively. In the fifth equality we
used from the proof of Proposition \ref{prop:vi_versions}~(2)$\Rightarrow$(3)
the fact that $E(1\ox b)\in \M(A)\ox A$. Symmetrically, using Proposition
\ref{prop:vi_versions}~(3) for $\epsilon'$ in the second equality, Proposition
\ref{prop:(a@1)E}~(2) in the fifth equality and Proposition
\ref{prop:vi_versions}~(1) for $\epsilon$ in the penultimate equality, 
\begin{eqnarray*}
(\epsilon\ox \epsilon')[(a\ox 1)\!\!\!&\!\!\!\Delta(b)\!\!\!&\!\!\!(1\ox c)]=
(\epsilon\ox \epsilon')[T_2(a\ox b)(1\ox c)]\\
&=&
(\epsilon\ox \epsilon' \ox \epsilon')[(T_2(a\ox b)\ox 1)(1\ox E)
(1\ox 1\ox c)]\\ 
&\stackrel{(iv)}=&
(\epsilon\ox \epsilon' \ox \epsilon')[(T_2(a\ox b)\ox 1)(E\ox 1)(1\ox E)(1\ox
1\ox c)] \\
&\stackrel{(v)}=& (\epsilon\ox \epsilon' \ox \epsilon')[(a\ox 1\ox 1)
(\overline{\Delta\ox \id})((b\ox 1)E)(1\ox 1\ox c)]\\
&=&(\epsilon\ox \epsilon' \ox \epsilon')[(T_2\ox \id)(a\ox (b\ox 1)E)
(1\ox 1\ox c)]\\
&\stackrel{(iii)' }=&
(\epsilon\ox \epsilon')[(ab\ox 1)E(1\ox c)]=
(\epsilon\ox \epsilon')T_2 (ab\ox c) \stackrel{(iii)'}=
\epsilon(abc),
\end{eqnarray*}
where the label (iii)' refers to the application of axiom (iii) to $\epsilon'$. 
So we conclude by the idempotency of $A$ that $\epsilon=\epsilon'$.
\end{proof}

Two main sources of examples of weak multiplier bialgebras are
regular weak multiplier Hopf algebras in \cite{VDaWa} and weak bialgebras
\cite{WHAI,Nill} (possessing units), as we shall see in the next two theorems.

\begin{theorem}\label{thm:reg_mwha>mwba}
If an idempotent algebra $A$ over a field with a non-degenerate multiplication
possesses a regular weak multiplier Hopf algebra structure in the sense of
\cite{VDaWa}, then $A$ is also a (regular) weak multiplier bialgebra via the
same structure maps.
\end{theorem}

\begin{proof} 
Axioms (i), (ii), (iii) and (v) in Definition \ref{def:mwba} are parts of the
definition of regular weak multiplier Hopf algebra in \cite{VDaWa}. Since $A$
is an idempotent algebra by assumption, the axioms $E(A\ox A)=T_1(A\ox A)$ and
$(A\ox A)E=T_2(A\ox A)$ in \cite{VDaWa} imply our axiom (iv). It
remains to prove that axiom (vi) holds true. By \cite[Proposition 2.3]{VDaWa},
for any weak multiplier Hopf algebra $A$ over a field, there exists a linear
map $R_1:A\ox A \to A\ox A$ such that $T_1R_1(a\ox b)=E(a\ox b)$ for all
$a,b\in A$. Then applying $\epsilon \ox \id$ to both sides and using the
counitality axiom (iii) in Definition \ref{def:mwba}, it follows that 
\begin{equation}\label{eq:muR1}
\mu R_1(a\ox b)=(\epsilon \ox \id)[E(a\ox b)]\qquad \forall a,b\in A.
\end{equation}
For any $a,b,c\in A$,
$$
T_1[(a\ox 1)R_1(b\ox c)]=
\Delta(a)(T_1R_1(b\ox c))
=\Delta(a)E(b\ox c)\stackrel{(iv)}=
\Delta(a)(b\ox c).
$$
Applying $\epsilon \ox \id$ to both sides and using the counitality axiom (iii)
and \eqref{eq:muR1}, 
$$
(\epsilon \ox \id)((1\ox a)E(b\ox c))=
(\epsilon \ox \id)(\Delta(a)(b\ox c)).
$$
So the first axiom in (vi) holds true. The assumption about regularity ---
which has not yet been used so far --- allows for a symmetric verification of
the second axiom in (vi). 
\end{proof}

For {\em arbitrary} weak multiplier Hopf algebras in \cite{VDaWa}, however, the 
second axiom in Definition \ref{def:mwba}~(vi) does not seem to
hold. Consequences of this will be discussed further in Section
\ref{sec:antipode}.

\begin{theorem}\label{thm:wba<>mwba}
For a unital algebra $A$ over a field, there is a bijective correspondence
between
\begin{itemize}
\item weak bialgebra structures on $A$,
\item and weak multiplier bialgebra structures on $A$.
\end{itemize}
\end{theorem}

\begin{proof} 
A unital algebra $A$ is clearly idempotent with a non-degenerate
multiplication, and its multiplier algebra $\M(A)$ coincides with $A$. So in
this case the axioms in Definition \ref{def:mwba} (i) become trivial
identities and any weak multiplier bialgebra structure on $A$ is regular. By
axioms (ii) and (iii), a weak multiplier bialgebra structure on $A$ is given
by a coassociative counital comultiplication $A\to A\ox A$ which is a
multiplicative map, and a compatible idempotent element of $A\ox A$. By the
uniqueness of the idempotent element $E$ obeying axiom (iv), it follows that
$E=\Delta(1)$. Then axiom (v) is the usual weak bialgebra axiom expressing the
weak comultiplicativity of the unit. By Proposition \ref{prop:vi_versions} and
Proposition \ref{prop:vi_regular}, axiom (vi) is equivalent to the usual weak
bialgebra axiom expressing the weak multiplicativity of the counit (cf. parts
(4) of the quoted propositions).
\end{proof}

Among weak bialgebras $A$ over a field, bialgebras are distinguished by the
equivalent properties that $\Delta(1)=1\ox 1$, or
$\epsilon(ab)=\epsilon(a)\epsilon(b)$ for all $a,b\in A$, or
$\pibarL(a)=\epsilon(a)1$ for all $a\in A$, or
$\pibarR(a)=\epsilon(a)1$ for all $a\in A$.
As shown in the next theorem, these properties (in appropriate forms) remain
equivalent also for a weak multiplier bialgebra $A$.

\begin{theorem}\label{thm:mba}
Let $A$ be a weak multiplier bialgebra over a field. The following
assertions are equivalent.
\begin{itemize}
\item[{(1)}] $E=1$ as elements of $\M(A\ox A)$.
\item[{(2)}] $\epsilon(ab)=\epsilon(a)\epsilon(b)$ for all $a,b\in A$.
\item[{(3)}] $\pibarL(a)=\epsilon(a)1$ as elements of $\M(A)$, for all $a\in A$.
\item[{(4)}] $\pibarR(a)=\epsilon(a)1$ as elements of $\M(A)$, for all $a\in A$.
\end{itemize}
\end{theorem}

\begin{proof}
(1)$\Rightarrow$(2). For any $a,b\in A$,
$\epsilon(a)\epsilon(b)\stackrel{(1)}=
(\epsilon\ox \epsilon)[(a\ox 1)E(1\ox b)]=
\epsilon(ab)$,
where the last equality follows by Proposition \ref{prop:vi_versions}~(3).

(2)$\Rightarrow$(1). Using Proposition \ref{prop:(a@1)E}~(1) in the first
equality, it follows for any $a,b,c,d\in A$ that
\begin{eqnarray*}
(ab \ox 1)E(1\ox cd)&=&
((\id \ox \pibarL)T_2(a\ox b))(1\ox cd)\\
&\stackrel{\eqref{eq:pibarL}}=&
(\id \ox \epsilon \ox \id)
[((\id \ox T_2)(T_2\ox \id)(a\ox b \ox c))(1\ox 1\ox d)]\\
&=&(\id \ox \epsilon \ox \id)
[(T_2(a\ox b)\ox 1)(1\ox T_1(c\ox d))]\\
&\stackrel{(2)}=&
(\id \ox \epsilon)T_2(a\ox b)\ox(\epsilon \ox \id)T_1(c\ox d)
\stackrel{(iii)}=
ab \ox cd,
\end{eqnarray*}
from which we conclude by the density of $A\ox A$ in $\M(A\ox A)$.

(1)$\Rightarrow$(3). For any $a,b\in A$,
$
b\pibarL(a)\stackrel{\eqref{eq:pibarL}}=
(\epsilon \ox \id)[(a\ox b)E]\stackrel{(1)}=
b\epsilon(a),
$
from which we conclude by the density of $A$ in $\M(A)$.

(1)$\Rightarrow$(4) follows symmetrically making use of \eqref{eq:pibarR}. 

(3)$\Rightarrow$(1). Using Proposition \ref{prop:(a@1)E}~(1) in the first
equality, it follows for any $a,b\in A$ that
$$
(ab \ox 1)E=
(\id \ox \pibarL)T_2(a\ox b)\stackrel{(3)}=
(\id \ox \epsilon)T_2(a\ox b)\ox 1\stackrel{(iii)}=
ab \ox 1,
$$
from which we conclude by the density of $A\ox A$ in $\M(A\ox A)$.

(4)$\Rightarrow$(1) follows symmetrically applying \eqref{eq:E(1@a)}.
\end{proof}

Whenever the equivalent conditions in Theorem \ref{thm:mba} hold for a 
weak multiplier bialgebra over a field, it would be most natural to term
it a {\em multiplier bialgebra}. Note, however, that this notion is different
from both notions in \cite{JaVe} and \cite{Tim} which were given the same
name. 

The next three examples do not belong to any of the previously
discussed classes. 

\begin{example}\label{ex:span}
Take a small category possibly with infinitely many objects and arrows. For a
fixed field, let $A$ be the vector space spanned by all of the arrows. It can
be equipped with the following non-unital algebra structure. For any arrows
$a$ and $b$, let their product be the composite arrow $ab$ if they are
composable and zero otherwise. Since the identity arrows of the category give
rise to local units, this is an idempotent algebra with a non-degenerate
multiplication. It can be equipped with the structure of a regular weak
multiplier bialgebra. The comultiplication takes an arrow $a$ to $a\ox a$
regarded as an element of the multiplier algebra $\M(A\ox A)$. The counit
takes each arrow to $1$. The idempotent element $E$ in $\M(A\ox A)$ is given
by $E(a\ox b)=a\ox b$ if the arrows $a$ and $b$ have equal targets and $E(a\ox
b)=0$ otherwise; and $(a\ox b)E=a\ox b$ if the arrows $a$ and $b$ have equal
sources and $(a\ox b)E=0$ otherwise. All these maps are then linearly extended. 
\end{example}

\begin{example}\label{ex:functional}
Take again a small category possibly with infinitely many objects and
arrows. For a fixed field, let $A$ be the vector space of linear functionals
of finite support on the vector space spanned by the arrow set $\mathcal
S$. It is a non-unital algebra via the pointwise multiplication
(i.e. $(fg)(a):=f(a)g(a)$ for every $f, g \in A$ and $a \in \mathcal{S}$). The
characteristic functions of the finite subsets of $\mathcal S$ serve as local
units for $A$ hence it is an idempotent algebra with a non-degenerate
multiplication. 

For any arrows $a$ and $b$ of common source, assume that there are only
finitely many arrows $c$ such that $ca=b$. Symmetrically, for any arrows $a$
and $b$ of common target, assume that there are only finitely many arrows $c$
such that $ac=b$. (These assumptions evidently hold for a groupoid.) Then $A$
carries the structure of a regular weak multiplier bialgebra. In terms
of the characteristic functions $\delta_p$ of the one element subsets $\{p\}$
of $\mathcal S$, the comultiplication $\Delta$ takes $f\in A$ to the
multiplier $\Delta (f)$ described by 
$$
\Delta(f)(g\ox h)=\sum_{p,q\in \mathcal S} g(p)h(q)f(pq)\delta_p\ox \delta_q=
(g\ox h)\Delta(f)
$$
for any $g,h\in A$. Note that in this sum there are only finitely many
non-zero terms since $g,h$ and $f$ have finite supports. The maps $T_j$
(for $j\in \{1,2,3,4 \}$) land in $A\ox A$ by the assumption that we made
about the set of arrows. The counit takes $f$ to the sum of the values $f(i)$
for the {\em identity} arrows $i$ (which contains finitely many non-zero terms
by assumption). The idempotent element $E$ in $\M(A\ox A)$ is given by
$$
E(g\ox h)=\sum_{\{p,q\in \mathcal S \ \mathrm{composable}\}} 
g(p)h(q)\delta_p\ox \delta_q=
(g\ox h)E\quad \forall g,h\in A.
$$
\end{example}

It was shown in \cite{VDaWa} that whenever the categories in the above
examples are group\-oids, then both constructions yield regular weak multiplier
Hopf algebras in the sense of \cite{VDaWa}. 

\begin{example}\label{ex:direct_sum}
In this example we show that any direct sum of weak multiplier bialgebras over
a field --- so in particular any infinite direct sum of weak bialgebras over a
field --- is a weak multiplier bialgebra.

For any index set $I$, consider a family of idempotent algebras $\{A_j\}_{j\in
I}$ over a field $k$ with non-degenerate multiplications $\mu_j$. Let $A : =
\bigoplus_{j \in I} A_j$ denote the direct sum vector space with the
inclusions $i_j:A_j\to A$ and the projections $p_j:A\to A_j$. The elements of
$A$ are the $I$-tuples $\underline a=\{a_j\in A_j\}_{j\in I}$
such that $a_j:= p_j(\underline a)$ is non-zero only for finitely many
indices $j\in I $. Clearly, $A$ can be equipped with the structure of an
idempotent algebra with a non-degenerate multiplication $\mu:\underline a \ox
\underline b \mapsto \underline a\,\underline b$, uniquely characterized
by $p_j(\underline a\, \underline b)=a_j b_j$, for any $\underline
a,\underline b\in A$ and $j\in I$ (so that $i_j$ becomes multiplicative as
well). 

The multiplier algebra of $A$ is isomorphic to the direct (in fact, Cartesian)
product $\prod_{j\in I} \M(A_j)$, regarded as a unital algebra via the factorwise
multiplication. (Its elements are $I$-tuples $\{\omega_j\in
\M(A_j)\}_{j\in I}$ without any restriction on the number of the non-zero
elements.) Indeed, $i_j(A_j)$ is an ideal in $A$ for any $j\in I$. Hence for
any $\omega \in \M(A)$, any $j\in I$, and any $a,b\in A_j$,
$$
\omega i_j(ab)=
\omega (i_j(a)i_j(b))=
(\omega i_j(a))i_j(b)
$$
is an element of $i_j(A_j)$. So by the idempotency of $A_j$, $\omega i_j(a)\in
i_j(A_j)$, and symmetrically, $i_j(a) \omega \in i_j(A_j)$, for any $j\in I$
and $a\in A_j$. This proves the existence of multipliers $\omega_j\in \M(A_j)$
such that 
$$
i_j(\omega_j a):=\omega i_j(a)\qquad \textrm{and}\qquad
i_j(a \omega_j):=i_j(a) \omega, \qquad \forall a\in A_j.
$$
Hence there is a map
\begin{equation}\label{eq:phi}
\varphi:\M(A)\to \prod_{j\in I} \M(A_j),\qquad
\omega \mapsto \{\omega_j\}_{j\in I}.
\end{equation}
It has the inverse $\{\omega_j\}_{j\in I}\mapsto \omega$ such that $p_j(\omega
\underline a)=\omega_j a_j$ and $p_j(\underline a \omega)=a_j \omega_j$ for
all $\underline a\in A$ and $j\in I$.

Let us take now two families $\{A_j\}_{j\in I}$ and $\{B_j\}_{j\in I}$ of
idempotent algebras with non-degenerate multiplications, together with a
family of multiplicative linear maps $\{\gamma_j:A_j\to \M(B_j)\}_{j\in I}$
and idempotent elements $\{e_j\in \M(B_j)\}_{j\in I}$ such that for all $j\in
I$, $\gamma_j(A_j)B_j=e_jB_j$ and $B_j\gamma_j(A_j)=B_je_j$. Then it follows
by Theorem \ref{thm:extend} that there exist unique multiplicative linear maps
$\{\overline \gamma_j:\M(A_j)\to \M(B_j)\}_{j \in I}$ extending $\gamma_j$
such that $\overline\gamma_j(1_j)=e_j$. Put $A:=\oplus_{j\in I} A_j$ and 
$B:=\oplus_{j\in I} B_j$ as before and in terms of the map \eqref{eq:phi}
define 
$$ 
e:=\varphi^{-1}(\{e_j\}_{j\in I})\in\M(B)
\quad \textrm{and}\quad
\gamma:A \to \M(B),\qquad
\underline a \mapsto \varphi^{-1}(\{\gamma_j(a_j)\}_{j\in I}).
$$
Then for all $j\in I$, $p_j(\gamma(A)B)=\gamma_j(A_j)B_j=e_jB_j=p_j(eB)$,
so that $\gamma(A)B=eB$. Symmetrically, $B\gamma(A)=Be$. Thus by Theorem
\ref{thm:extend}, $\gamma$ extends to a unique multiplicative linear map
$\overline \gamma:\M(A)\to \M(B)$ such that $\overline
\gamma(1)=e$. Explicitly, 
\begin{equation}\label{eq:gamma_bar}
p_j(\overline \gamma(\omega)\underline b)=\overline \gamma_j(\omega_j)b_j
\qquad\textrm{and}\qquad
p_j(\underline b \overline \gamma(\omega))=b_j\overline \gamma_j(\omega_j),
\end{equation}
for all $j\in I$, $\underline b\in B$ and $\omega \in \M(A)$.

Assume next that for all $j\in I$, $A_j$ carries a weak multiplier bialgebra
structure with comultiplication $\Delta_j:A_j\to \M(A_j\ox A_j)$, counit
$\epsilon_j:A_j\to k$ and idempotent element $E_j\in \M(A_j\ox A_j)$. Since
$A\ox A\cong \oplus_{j,l\in I} A_j\ox A_l$, its multiplier algebra $\M(A\ox
A)$ is isomorphic to $\prod_{j,l\in I}\M(A_j\ox A_l)$. Hence $\M(A\ox
A)$ has a non-unital subalgebra $\prod_{j\in I}\M(A_j\ox A_j)$. In terms of
the map \eqref{eq:phi}, we define 
$$
\begin{array}{rll}
\Delta:&A\to \prod_{j\in I} \M(A_j\ox A_j)\subset \M(A\ox A),\qquad
&\underline a \mapsto \varphi^{-1}(\{\Delta_j(a_j)\}_{j\in I})\\
\epsilon:&A \to k,\qquad
&\underline a \mapsto \sum_{j\in I}\epsilon_j(a_j)\\
E\in& \prod_{j\in I}\M(A_j\ox A_j)\subset \M(A\ox A),\qquad 
&E:=\varphi^{-1}(\{E_j\}_{j\in I}).
\end{array}
$$
Note that the counit $\epsilon$ is well-defined because the sum has only
finitely many non-zero terms. This equips $A$ with the structure of a weak
multiplier bialgebra. Moreover, if $A_j$ is a regular weak multiplier
bialgebra for all $j\in I$, then so is the direct sum $A=\oplus_{j\in I} A_j$.
\end{example}

\section{The base algebras}\label{sec:base}

Let $A$ be a weak multiplier bialgebra over a field. The aim of this section
is to study the properties of the maps 
$$
\pibarL:A\to \M(A),\qquad a\mapsto (\epsilon \ox \id)((a\ox 1)E)
$$ 
in \eqref{eq:pibarL} and 
$$
\pibarR:A\to \M(A),\qquad a\mapsto (\id \ox \epsilon)(E(1\ox a))
$$ 
in \eqref{eq:pibarR} in a remarkable analogy with the unital case. Their
images in $\M(A)$ are termed as the {\em base algebras} (they are indeed
subalgebras of $\M(A)$ by Lemma \ref{lem:mod_map} below), and they will be
investigated further in the next Section \ref{sec:sF}. 

\begin{lemma}\label{lem:counit_exp}
For any weak multiplier bialgebra $A$ over a field, and any $a,b\in A$, 
$$
\epsilon(\pibarL(a)b)=\epsilon(ab)\qquad \textrm{and}\qquad 
\epsilon(a\pibarR(b))=\epsilon(ab).
$$
\end{lemma}

\begin{proof} For any $a,b\in A$,
$\epsilon(\pibarL(a)b)\stackrel{\eqref{eq:pibarL}}=
(\epsilon \ox \epsilon)T_2(a\ox b)\stackrel{(iii)}=
\epsilon(ab)$.
The other equality is proven symmetrically.
\end{proof}

\begin{lemma}\label{lem:cond_exp}
For any weak multiplier bialgebra $A$ over a field, and any $a,b\in A$, 
$$
\pibarL(\pibarL(a)b)=\pibarL(ab)\qquad \textrm{and}\qquad 
\pibarR(a\pibarR(b))=\pibarR(ab).
$$
\end{lemma}

\begin{proof}
Using Lemma \ref{lem:counit_exp} in the second equality,
$$
\pibarL(\pibarL(a)b)=
(\epsilon \ox \id)[(\pibarL(a)b\ox 1)E]=
(\epsilon \ox \id)[(ab\ox 1)E]=
\pibarL(ab),
$$
for any $a,b\in A$. The other equality is proven symmetrically. 
\end{proof}

\begin{lemma}\label{lem:coproduct}
For any weak multiplier bialgebra $A$ over a field, and any $a\in A$, 
\begin{eqnarray*}
\overline\Delta\pibarL(a)=&(\pibarL(a)\ox 1)E&=E(\pibarL(a)\ox 1)
\qquad \textrm{and}\\
\overline\Delta\pibarR(a)=&(1\ox \pibarR(a))E&=E(1\ox \pibarR(a)).
\end{eqnarray*}
\end{lemma}

\begin{proof}
For any $a\in A$, $(a\ox 1)E\in A\ox \M(A)$ by Proposition
\ref{prop:(a@1)E}~(2). Hence
\begin{eqnarray*}
\overline\Delta\pibarL(a)&=&
\overline\Delta(\epsilon \ox \id)[(a\ox 1)E]=
(\epsilon \ox \id)(\id \ox \overline\Delta)[(a\ox 1)E]\\
&\stackrel{(v)}=&
(\epsilon \ox \id)[(a\ox 1\ox 1)(E\ox 1)(1\ox E)]\\
&=& ((\epsilon \ox \id)[(a\ox 1)E]\ox 1)E=
(\pibarL(a)\ox 1)E.
\end{eqnarray*}
All of the remaining equalities are verified analogously. 
\end{proof}

\begin{lemma}\label{lem:mod_map}
For any weak multiplier bialgebra $A$ over a field, and any $a,b\in A$, 
$$
\pibarL(a\pibarL(b))=
\pibarL(a)\pibarL(b)\qquad \textrm{and}\qquad 
\pibarR(\pibarR(a)b)=
\pibarR(a)\pibarR(b).
$$
\end{lemma}

\begin{proof}
For any $\psi\in \M(A)$ such that $\overline \Delta(\psi)=(\psi \ox 1)E$, and
for any $a,b\in A$,
\begin{eqnarray*}
\pibarL(a\psi)b&\stackrel{\eqref{eq:pibarL}}=&
(\epsilon \ox \id)T_2(a\psi\ox b)\stackrel{(iv)}=
(\epsilon \ox \id)[(a\psi\ox 1)E\Delta(b)]\\
&=&(\epsilon \ox \id)[(a\ox 1)\Delta(\psi b)]=
(\epsilon \ox \id)T_2(a\ox \psi b)\stackrel{\eqref{eq:pibarL}}=
\pibarL(a)\psi b,
\end{eqnarray*}
where in the third equality we used the assumption about $\psi$ and the
multiplicativity of $\overline\Delta$. Since $A$ is dense in
$\M(A)$, this proves $\pibarL(a\psi)=\pibarL(a)\psi$. We conclude by Lemma
\ref{lem:coproduct} that $\pibarL(a\pibarL(b))=\pibarL(a)\pibarL(b)$, for any
$a,b\in A$. The other equality is proven symmetrically.
\end{proof}

\begin{lemma}\label{lem:commute}
For any weak multiplier bialgebra $A$ over a field, and any $a,b\in A$, 
$$
\pibarR(a)\pibarL(b)=\pibarL(b)\pibarR(a).
$$
\end{lemma}

\begin{proof}
For any $a,b\in A$, 
\begin{eqnarray*}
\pibarR(a)\pibarL(b)&=&
(\id \ox \epsilon)[E(1\ox a)](\epsilon \ox \id)[(b\ox 1)E]\\
&=&(\epsilon \ox \id \ox \epsilon)[(1\ox E)(b\ox 1\ox a)(E\ox 1)]\\
&=&
(\epsilon \ox \id \ox \epsilon)[(b\ox 1\ox 1)(1\ox E)(E\ox 1)(1\ox 1\ox a)]\\
&\stackrel{(v)}=&
(\epsilon \ox \id \ox \epsilon)[(b\ox 1\ox 1)(E\ox 1)(1\ox E)(1\ox 1\ox a)]\\
&=&(\epsilon \ox \id)[(b\ox 1)E](\id \ox \epsilon)[E(1\ox a)]=
\pibarL(b)\pibarR(a).
\end{eqnarray*}
\end{proof}

\begin{lemma}\label{lem:pi_(a)}
For any weak multiplier bialgebra $A$ over a field, and for any
$a,b,c,d\in A$,
$$
(ab\ox 1)((\pibarR\ox \id)T_1(c\ox d))=
((\id \ox \pibarL)T_2(a\ox b))(1\ox cd).
$$
\end{lemma}

\begin{proof}
Both expressions in the claim are equal to $(ab\ox 1)E(1\ox cd)$, see
Proposition \ref{prop:(a@1)E}~(1) and the proof of Proposition
\ref{prop:vi_versions}~(2)$\Rightarrow$(3).
\end{proof}

Whenever $A$ is a {\em regular} weak multiplier bialgebra over a field, the
above considerations can be repeated in the opposite algebra. That is, we can
define the maps 
\begin{eqnarray}
\label{eq:piL}
&\piL:A\to \M(A),\qquad 
&a\mapsto (\epsilon \ox \id)(E(a\ox 1))\\
\label{eq:piR}
&\piR:A\to \M(A),\qquad 
&a\mapsto (\id \ox \epsilon)((1\ox a)E).
\end{eqnarray}
They obey the following properties, for all $a,b,c,d\in A$.
\begin{eqnarray}
\label{eq:piR(a)b}
&&\piR(a)b=(\id \ox \epsilon)T_3(b\ox a)\ \textrm{and}\ 
b\piL(a)=(\epsilon\ox \id)T_4(a\ox b).\\
\label{eq:E(a@1)}
&&(1\ox ab)E=(\piR\ox \id)T_3(b\ox a)\ \textrm{and}\ 
E(ab\ox 1)=(\id \ox \piL)T_4(b\ox a),\\
\label{eq:counit_exp_reg}
&&\epsilon(a\piL(b))=\epsilon(ab)\ \textrm{and}\ 
\epsilon(\piR(a)b)=\epsilon(ab),\\
\label{eq:cond_exp_reg}
&&\piL(a\piL(b))=\piL(ab)\ \textrm{and}\ 
\piR(\piR(a)b)=\piR(ab),\\
\label{eq:coproduct_reg}
&&\overline\Delta\piL(a)=(\piL(a)\ox 1)E=E(\piL(a)\ox 1)
\ \textrm{and}\\
&&\overline\Delta\piR(a)=(1\ox \piR(a))E=E(1\ox \piR(a)),\nonumber\\
\label{eq:mod_map_reg}
&&\piL(\piL(a)b)=\piL(a)\piL(b)\ \textrm{and}\
\piR(a\piR(b))=\piR(a)\piR(b),\\
\label{eq:commute_reg}
&&\piL(a)\piR(b)=\piR(b)\piL(a),\\
\label{eq:pi_identity_reg}
&&((\piR\ox \id)T_3(a\ox b))(cd\ox 1)=
(1\ox ba)((\id \ox \piL)T_4(d\ox c)).
\end{eqnarray}
\smallskip

In a weak bialgebra possessing a unit, the above maps behave as generalized
counits: $\mu(\piL\ox \id)\Delta=\id=\mu(\id \ox \piR)\Delta$ and
$\mu^{\op}(\pibarL\ox \id)\Delta=\id = \mu^{\op}(\id \ox \pibarR)\Delta$. In the
following lemma these identities are generalized to weak multiplier
bialgebras. 

\begin{lemma}\label{lem:counital_maps}
For a regular weak multiplier bialgebra $A$ over a field, the following
equalities hold. 
\begin{itemize}
\item[{(1)}] $\mu^{\op}(\pibarL\ox \id)T_3=\mu^\op$.
\item[{(2)}] $\mu^{\op}(\id \ox \pibarR)T_4=\mu^\op$.
\item[{(3)}] $\mu(\piL\ox \id)T_1=\mu$.
\item[{(4)}] $\mu(\id \ox \piR)T_2=\mu$.
\end{itemize}
\end{lemma}

\begin{proof}
We spell out the proof only for (1), all other assertions are proven
symmetrically. For any $a,b,c\in A$, 
\begin{eqnarray*}
(\mu^{\op}(\pibarL\ox \id)T_3(a\ox b))c&=&
\mu^\op[((\pibarL\ox \id)T_3(a\ox b))(c\ox 1)]\\
&\stackrel{\eqref{eq:pibarL}}=&
(\epsilon \ox \id)[(1\ox b)\Delta(a)\Delta(c)]=
(\epsilon \ox \id)T_3(ac\ox b)\stackrel{(iii)}=
bac,
\end{eqnarray*}
so we conclude by non-degeneracy of the multiplication.
\end{proof}

\begin{lemma}\label{lem:counit_inv}
For a regular weak multiplier bialgebra $A$ over a field, and any $a,b\in
A$,
$$
\epsilon(\piL(a)b)=\epsilon(\pibarR(b)a)\quad \textrm{and}\quad
\epsilon(a\piR(b))=\epsilon(b\pibarL(a)).
$$
\end{lemma}

\begin{proof}
In light of \eqref{eq:piL} and \eqref{eq:pibarR}, respectively, both sides of
the first equality are equal to $(\epsilon\ox \epsilon)[E(a\ox b)]$ hence also
to each other. The second equality follows symmetrically.
\end{proof}

\begin{lemma}\label{lem:E_balanced}
For a regular weak multiplier bialgebra $A$ over a field, and any $a\in
A$, the following equalities hold. 
$$
(\piR(a)\ox 1)E=(1\ox \pibarL(a))E\qquad \textrm{and}\qquad 
 E(\pibarR(a)\ox 1)=E(1\ox \piL(a)).
$$
\end{lemma}

\begin{proof}
By axiom (iv) in Definition \ref{def:mwba}, the first equality in the claim is
equivalent to 
\begin{equation}\label{eq:E_balanced}
(\piR(a)\ox 1)T_1(b\ox c)=(1\ox \pibarL(a))T_1(b\ox c), 
\quad \forall b,c\in A. 
\end{equation}
Using the identities $c\piR(a)b=(\id \ox \epsilon)[(c\ox a)\Delta(b)]=(\id \ox
\epsilon)[(1\ox a)T_2(c\ox b)]$ and
$$
(\epsilon \ox \id)[(a\ox 1)T_1(b\ox c)]
=(\epsilon \ox \id)[T_2(a\ox b)(1\ox c)]\stackrel{\eqref{eq:pibarL}}=
\pibarL(a)bc
$$
(for any $a,b,c\in A$) in the first and the penultimate equalities,
respectively, one computes
\begin{eqnarray*}
(d\piR(a)\ox 1)T_1(b\ox c)&=&
(\id \ox \epsilon \ox \id)
[(1\ox a\ox 1)((T_2\ox \id)(\id \ox T_1)(d\ox b\ox c))]\\
&\stackrel{(ii)}=&
(\id \ox \epsilon \ox \id)
[(1\ox a\ox 1)((\id \ox T_1)(T_2\ox \id)(d\ox b\ox c))]\\
&=&(1\ox \pibarL(a))(d\ox 1)\Delta(b)(1\ox c)\\
&=&(d\ox \pibarL(a))T_1(b\ox c),
\end{eqnarray*}
for all $a,b,c,d\in A$. So we conclude by the non-degeneracy of the
multiplication that \eqref{eq:E_balanced}, and hence $(\piR(a)\ox 1)E=(1\ox
\pibarL(a))E$, holds for all $a\in A$. The other equality is proven
symmetrically. 
\end{proof}

\begin{lemma}\label{lem:pi_(ef)}
For any regular weak multiplier bialgebra $A$ over a field, and for any
$a,b,c,d\in A$,
\begin{eqnarray*}
&&(1\ox ab)((\id \ox \pibarR)T_4(c\ox d))=
((\piR\ox \id)T_2^{\op}(a\ox b))(dc\ox 1)\ \quad \textrm{and}\\
&&((\pibarL\ox \id)T_3(a\ox b))(cd\ox 1)=
(1\ox ba)((\id \ox \piL)T_1^{\op}(c\ox d)),
\end{eqnarray*}
where $T_1^{\op}=\mathsf{tw}\, T_1$ and $T_2^{\op}=\mathsf{tw}\, T_2$.
\end{lemma}

\begin{proof}
For any $a,b,c,d\in A$,
\begin{eqnarray*}
(1\ox 1\ox a)((\id \ox T_1^\op\,\mathsf{tw})
\!\!\!\!\!\!\!&&\!\!\!\!\!\!\!(T_4\ox \id)(c\ox d\ox b))\\
&=&(1\ox 1\ox a)(1\ox \Delta^\op(b))(\Delta(d)\ox 1)(c\ox 1\ox 1)\\
&=&((T_3\ox \id)(\id \ox T_2^\op)(d\ox a\ox b))(c\ox 1\ox 1).
\end{eqnarray*}
Applying $\id \ox \epsilon \ox \id$ to both sides, and using the identities
\eqref{eq:pibarR} and \eqref{eq:piR(a)b}, we obtain the first equality in the
claim. The second equality is proven symmetrically.
\end{proof}

\begin{lemma}\label{lem:bar_mod_map}
Let $A$ be a regular weak multiplier bialgebra over a field. For any $a,b\in
A$,
\begin{eqnarray*}
&&\piR(a\pibarR(b))=\piR(a)\pibarR(b)=\pibarR(\piR(a)b)\quad \textrm{and}\\
&&\piL(\pibarL(a)b)=\pibarL(a)\piL(b)=\pibarL(a\piL(b)).
\end{eqnarray*}
\end{lemma}

\begin{proof}
Applying the multiplicativity of $\overline \Delta:\M(A)\to \M(A\ox A)$,
Lemma \ref{lem:coproduct} and axiom (iv) in the second equality, 
$$
T_3(\pibarR(a)b\ox c)=
(1\ox c)\Delta(\pibarR(a)b)=
(1\ox c\pibarR(a))\Delta(b)=
T_3(b\ox c\pibarR(a)),
$$
for any $a,b,c\in A$. Hence using \eqref{eq:piR(a)b},
$$
\piR(a)\pibarR(b)c=
(\id \ox \epsilon)T_3(\pibarR(b)c\ox a)=
(\id \ox \epsilon)T_3(c\ox a\pibarR(b)) =
\piR(a\pibarR(b))c.
$$
This proves the first equality in the claim and all other equalities are
checked analogously. 
\end{proof}

\begin{lemma}\label{lem:piRonL}
Let $A$ be a regular weak multiplier bialgebra over a field. For any $a,b\in
A$, the following hold.
\begin{eqnarray*}
&\piR(a\pibarL(b))=\piR(b)\piR(a)\qquad
&\piL(\pibarR(a)b)=\piL(b)\piL(a)\\
&\pibarR(\piL(b)a)=\pibarR(a)\pibarR(b)\qquad
&\pibarL(b\piR(a))=\pibarL(a)\pibarL(b).
\end{eqnarray*}
\end{lemma}

\begin{proof}
We only prove the first assertion, all of the remaining ones are proven
symmetrically. For any $a,b,c\in A$, 
$$
\piR(a\pibarL(b))c\stackrel{\eqref{eq:piR(a)b}}=
(\id \ox \epsilon)T_3(c\ox a\pibarL(b))=
\piR(b)(\id \ox \epsilon)T_3(c\ox a)\stackrel{\eqref{eq:piR(a)b}}=
\piR(b)\piR(a)c,
$$
where the second equality follows by axiom (iv) in Definition \ref{def:mwba}
and Lemma \ref{lem:E_balanced}. So we conclude by the density of $A$ in
$\M(A)$. 
\end{proof}

In the following theorem, for any vector spaces $V$ and $W$,
$\mathsf{Lin}(V,W)$ denotes the vector space of linear maps $V\to W$.

\begin{theorem}\label{thm:full}
For a regular weak multiplier bialgebra $A$ over a field $k$, the following
assertions are equivalent to each other.
\begin{itemize}
\item[{(1)}] The comultiplication is {\em right full} in the sense that
$$
\langle (\id \ox \omega)T_1(a\ox b)\ |\ a,b\in A, \omega\in
\mathsf{Lin}(A,k)\rangle =A. 
$$
\item[{(2)}] The comultiplication is {\em right full} in the sense that
$$
\langle (\id \ox \omega)T_3(a\ox b)\ |\ 
a,b\in A, \omega\in \mathsf{Lin}(A,k)\rangle =A.
$$
\item[{(3)}] $\langle (\id \ox \epsilon)T_1(a\ox b)\ |\ a,b\in A\rangle =A$. 
\item[{(4)}] $\langle (\id \ox \epsilon)T_3(a\ox b)\ |\ a,b\in A\rangle =A$. 
\item[{(5)}] $\{ \piR(a)\ |\ a\in A\}=\{ \pibarR(a)\ |\ a\in A\}$. 
\end{itemize}
The following assertions are equivalent to each other, too.
\begin{itemize}
\item[{(1)'}] The comultiplication is {\em left full} in the sense that
$$
\langle (\omega \ox \id)T_2(a\ox b)\ |\ 
a,b\in A, \omega\in \mathsf{Lin}(A,k)\rangle =A.
$$
\item[{(2)'}] The comultiplication is {\em left full} in the sense that
$$
\langle (\omega \ox \id)T_4(a\ox b)\ |\ 
a,b\in A, \omega\in \mathsf{Lin}(A,k)\rangle =A.
$$
\item[{(3)'}] $\langle (\epsilon \ox \id)T_2(a\ox b)\ |\ a,b\in A\rangle =A$. 
\item[{(4)'}] $\langle (\epsilon \ox \id)T_4(a\ox b)\ |\ a,b\in A\rangle =A$. 
\item[{(5)'}] $\{ \piL(a)\ |\ a\in A\}=\{ \pibarL(a)\ |\ a\in A\}$. 
\end{itemize}
\end{theorem}

\begin{proof}
We only prove the equivalence of the first five assertions. The equivalence of
the second quintuple is proven analogously.

(1)$\Leftrightarrow$(2) is proven in \cite[Lemma 1.11]{VDaWa:Banach}.

(3)$\Rightarrow$(1) and (4)$\Rightarrow$(2) are trivial. 

(1) and (2)$\Rightarrow$(5). For any $a,b\in A$, it follows by
\eqref{eq:E(a@1)} and \eqref{eq:E(1@a)} that 
$$
(\piR\ox \id)T_1(a\ox b)=(1\ox a)E(1\ox b)=(\pibarR\ox \id)T_3(b\ox a).
$$
Then (1) implies 
$$
\{ \piR(a)\ |\ a\in A\}=\langle (\id\ox \omega)[(1\ox a)E(1\ox b)]\ |\ 
a,b\in A, \omega\in \mathsf{Lin}(A,k)\rangle 
$$
and (2) implies 
$$
\{ \pibarR(a)\ |\ a\in A\}=\langle (\id\ox \omega)[(1\ox a)E(1\ox b)]\ |\ 
a,b\in A, \omega\in \mathsf{Lin}(A,k)\rangle ,
$$
proving the claim.

(5)$\Rightarrow$(3). By Lemma \ref{lem:counital_maps}~(4) and the idempotency
of the algebra $A$, $\langle a\piR(b)\ |\ a,b\in A\rangle =A$. Hence by
(5), 
$$
A=
\langle a\pibarR(b)\ |\ a,b\in A\rangle \stackrel{\eqref{eq:pibarR}}=
\langle (\id \ox \epsilon)T_1(a\ox b)\ |\ a,b\in A\rangle .
$$

(5)$\Rightarrow$(4) is proven symmetrically.
\end{proof}

\section{Firm separable Frobenius structure of the base algebras}\label{sec:sF}

In a weak bialgebra with a unit, the coinciding range of the maps $\pibarL$
and $\piL$, and the coinciding range of $\pibarR$ and $\piR$ in the previous
section, carry the structures of anti-isomorphic {\em separable Frobenius
algebras} (with units). The aim of this section is to see that in a regular
weak multiplier bialgebra with a left and right full comultiplication, the
base algebras still carry the structures of anti-isomorphic coseparable and
co-Frobenius coalgebras. Consequently, they are firm Frobenius algebras in the
sense of \cite{BoGT:fF}.  

It follows immediately from Lemma \ref{lem:mod_map} that for any weak
multiplier bialgebra $A$, the ranges of $\pibarL$ and of $\pibarR$ are
non-unital subalgebras of the multiplier algebra $\M(A)$. We turn to proving
that --- whenever $A$ is regular with a left and right full comultiplication
--- they carry coalgebra structures as well. First we look for the candidate
counit.  

\begin{proposition}\label{prop:base_counit}
Let $A$ be a regular weak multiplier bialgebra over a field $k$ with a right
full comultiplication. Then the counit $\epsilon$ determines a linear map 
$$
\varepsilon:\{\piR(a)\ |\ a\in A\}\to k,\qquad\qquad 
\piR(a)\mapsto \epsilon(a).
$$
\end{proposition}

\begin{proof}
In order to see that $\varepsilon$ is a well-defined linear map, we need to
show that $\piR(p)=0$ implies $\epsilon(p)=0$, for any $p\in A$. Since $A$ is
an idempotent algebra by assumption, we can write any element $p$ of $A$ as a
linear combination $\sum_i a^ib^i$ in terms of finitely many elements 
$a^i,b^i\in A$. So omitting throughout the summation symbol for brevity, it is
enough to prove that $\piR(a^ib^i)=0$ implies $\epsilon(a^ib^i)=0$, for any
finite set of elements $a^i,b^i\in A$. If $\piR(a^ib^i)=0$, then 
\begin{eqnarray*} 
0 \stackrel{\eqref{eq:counit_exp_reg}}= 
(\id \ox \epsilon)[(1\ox a^ib^i)((\id \ox \pibarR)
\!\!\!\!&\!\!\!\!T_4\!\!\!\!&\!\!\!\!(h\ox d))(1\ox c)]\\
&=&(\epsilon \ox \id)[((\id \ox \piR)T_2(a^i\ox b^i))(c\ox dh)]\\
&=&(\epsilon \ox \id)[((\id \ox \piR)T_2(a^i\ox b^i))
(\pibarR(c)\ox 1)]dh,
\end{eqnarray*}
for all $c,d,h\in A$. In the
second equality above, we used the first statement in Lemma \ref{lem:pi_(ef)}
and the third equality follows by Lemma \ref{lem:counit_exp}. By Theorem
\ref{thm:full}, the vector subspaces $\piR(A):=\{\piR(a)\ |\ a\in A\}$ and
$\pibarR(A):=\{\pibarR(a)\ |\ a\in A\}$ of $\M(A)$ coincide. So by the
density of $A$ in $\M(A)$, the map 
$$
\piR(A)\to \piR(A),\qquad \pibarR(c) \mapsto 
(\epsilon \ox \id)[((\id \ox \piR)T_2(a^i\ox b^i))
(\pibarR(c) \ox 1)]
$$
is the zero map. Introduce the notation $T_2(a^i\ox b^i)=:a'\ox b'$
(allowing for implicit summation on both sides of the equality). Using that
the evident map $\mathsf{Lin}(\piR(A),k)\ox \piR(A)\to
\mathsf{Lin}(\piR(A),\piR(A))$, $\Phi\ox x\mapsto \Phi(-)x$ is injective, we 
conclude that 
$$
\epsilon(a'-)\ox \piR(b')\in \mathsf{Lin}(\piR(A),k)\ox \piR(A)
$$
is equal to zero. Applying to it the evaluation map $\mathsf{Lin}(\piR(A),k)\ox
\piR(A)\to k$, $\Phi \ox x\mapsto \Phi(x)$, and using Lemma
\ref{lem:counital_maps}~(4), we proved that
$$
\epsilon(a'\piR(b'))=\epsilon\mu(\id \ox \piR)T_2(a^i\ox b^i)= \epsilon(a^ib^i)
$$
is equal to zero as needed.
\end{proof}

For the construction of the comultiplication on $\piR(A)$, the following
technical lemma is needed. 

\begin{lemma}\label{lem:F_well-defd}
Let $A$ be a regular weak multiplier bialgebra over a field. For any $a,b\in
A$, the following assertions hold.
\begin{itemize}
\item[{(1)}] The element $(\id \ox \pibarR)T_4(a\ox b)$ of $A\ox \M(A)$
 depends on $a$ and $b$ only through the product $ba$.
\item[{(2)}] The element $(\id \ox \piR)T_2(a\ox b)$ of $A\ox \M(A)$
 depends on $a$ and $b$ only through the product $ab$.
\end{itemize}
\end{lemma}

\begin{proof}
We only prove part (1), part (2) follows analogously.
Applying twice the first identity in Lemma \ref{lem:pi_(ef)}, for all
$a,b,c,d,f,g\in A$
\begin{eqnarray*}
(1\ox cd)((\id \ox \pibarR)T_4(a\ox b))(f\ox g)&=&
((\piR\ox \id)T_2^\op(c\ox d))(baf\ox g)\\
&=&(1\ox cd)((\id \ox \pibarR)T_4(f\ox ba)(1\ox g).
\end{eqnarray*}
So if $ba=b'a'$ for some $a,b,a',b'\in A$, then for all $f,g\in A$, 
\begin{eqnarray*}
((\id \ox \pibarR)T_4(a\ox b))(f\ox g)&=&
((\id \ox \pibarR)T_4(f\ox ba))(1\ox g)\\
&=&((\id \ox \pibarR)T_4(f\ox b'a'))(1\ox g)\\
&=&((\id \ox \pibarR)T_4(a'\ox b'))(f\ox g),
\end{eqnarray*}
proving $(\id \ox \pibarR)T_4(a\ox b)=(\id \ox \pibarR)T_4(a'\ox b')$.
\end{proof}

\begin{proposition}\label{prop:F}
For a regular weak multiplier bialgebra $A$ over a field, the following
assertions hold. 
\begin{itemize}
\item[{(1)}] The maps $A\ox A \to A \ox A$, 
\begin{eqnarray*}
&&a\ox bc\mapsto ((\pibarR\ox \id) T_4^\op(c\ox b))(a\ox 1)
\quad \textrm{and}\\
&&ab\ox c \mapsto (1\ox c)((\id \ox \piR)T_2(a\ox b))
\end{eqnarray*}
(where $T_4^\op=\mathsf{tw}T_4$) determine an element of $\M(A\ox A)$, to be
denoted by $F$. 
\item[{(2)}] For any element $a\in A$, and $F\in \M(A\ox A)$ as in (1), 
$(\piR(a)\ox 1)F$ and $F(1\ox \piR(a))$ are equal elements of $\piR(A)\ox
\piR(A)$, to be denoted by $\delta\sqcap^R(a)$.
\item[{(3)}] The map 
$$
\delta:\piR(A)\to \piR(A)\ox \piR(A),\qquad 
\piR(a)\mapsto (\piR(a)\ox 1)F= F(1\ox \piR(a))
$$
provides a $\piR(A)$-bimodule section of the multiplication in $\piR(A)$. 
\item[{(4)}] The map $\delta:\piR(A)\to \piR(A)\ox \piR(A)$ in part (3) is a
coassociative comultiplication.
\end{itemize}
\end{proposition}

\begin{proof}
(1). Both maps in the claim are well-defined by Lemma \ref{lem:F_well-defd}
and they define a multiplier by the first statement in Lemma
\ref{lem:pi_(ef)}. 

(2). Centrality of $F$ in the $\piR(A)$-bimodule $\M(A\ox A)$ follows by the
following computation, for all $a,b,c,d\in A$.
\begin{eqnarray*}
(bc\ox d)(\piR(a)\ox 1)F&=&
(1\ox d)((\id \ox \piR)T_2(b\ox c\piR(a)))\\
&=&(1\ox d)((\id \ox \piR)(T_2(b\ox c)\overline \Delta\piR(a)))\\
&=&(1\ox d)((\id \ox \piR)(T_2(b\ox c)(1\ox \piR(a))))\\
&=&(1\ox d)((\id \ox \piR)T_2(b\ox c))(1\ox \piR(a))\\
&=&(bc\ox d)F(1\ox \piR(a)).
\end{eqnarray*}
The second equality follows by the explicit form of $T_2$ and the
multiplicativity of $\overline \Delta$. In the third equality we used that by
\eqref{eq:coproduct_reg}, $\overline \Delta \piR(a)=E(1\ox \piR(a))$, and by
axiom (iv) in Definition \ref{def:mwba}, $T_2(b\ox c)E=T_2(b\ox c)$. The
fourth equality follows by \eqref{eq:mod_map_reg}. The stated elements belong
to $\piR(A)\ox \piR(A)$ by the following reasoning. For any $a,b,d,f\in A$,
denote $a'\ox b':=T_2(a\ox b)$ and $f'\ox d':=T_4(f\ox d)$ (allowing for
implicit summation). Then for all $a,b,c,d,f\in A$, 
\begin{eqnarray*}
(\piR(ab)\ox 1)F(c\ox df)&=&
\piR(ab)\pibarR(d')c\ox f'=
\piR(ab\pibarR(d'))c\ox f'\\
&=&\piR(a')c\ox \piR(b')df=
(\piR\ox \piR)T_2(a\ox b)(c\ox df).
\end{eqnarray*}
The second equality follows by Lemma \ref{lem:bar_mod_map} and the third one
follows by the first assertion in Lemma \ref{lem:pi_(ef)}. This proves 
\begin{equation}\label{eq:delta}
\delta\piR(ab)=(\piR\ox \piR)T_2(a\ox b), \quad \forall a,b\in A,
\end{equation}
so by the idempotency of $A$, $\delta\piR(a)\in \piR(A)\ox \piR(A)$, for all
$a\in A$.

(3). In view of \eqref{eq:delta}, for any $a,b\in A$,
$$\mu\delta\piR(ab)=
\mu(\piR\ox \piR)T_2(a\ox b)=
\piR\mu(\id \ox \piR)T_2(a\ox b)=
\piR(ab).
$$ 
The second identity follows by \eqref{eq:mod_map_reg} and the last one does by
Lemma \ref{lem:counital_maps}~(4).

(4). Let us use a Sweedler type index notation $\delta(r)=:r_1\ox r_2$ for any
$r\in \piR(A)$, where implicit summation is understood. It follows by part (3)
that $\piR(A)$ is an idempotent algebra. So the coassociativity of $\delta$
follows by
$$
(\delta \ox \id)\delta(sr)=
\delta(sr_1)\ox r_2=
s_1\ox s_2r_1\ox r_2=
s_1\ox \delta(s_2r)=
(\id \ox \delta)\delta(sr),
$$ 
for all $s,r\in \piR(A)$. In the first and the penultimate equalities we used
that $\delta$ is a morphism of left $\piR(A)$-modules and in the second and
the last equalities we used that it is a morphism of right $\piR(A)$-modules.
\end{proof}

\begin{theorem}\label{thm:base_coalg}
Let $A$ be a regular weak multiplier bialgebra over a field $k$ with a right
full comultiplication. Then $\piR(A)$ is a coalgebra via the counit
$\varepsilon:\piR(A)\to k$ in Proposition \ref{prop:base_counit} and the
comultiplication $\delta:\piR(A)\to \piR(A)\ox \piR(A)$ in Proposition
\ref{prop:F}~(3). 
\end{theorem}

\begin{proof}
The map $\delta$ is a coassociative comultiplication by Proposition
\ref{prop:F}~(4). It remains to prove its counitality. For any $a,b\in A$, 
$$
(\id \ox \varepsilon)\delta \piR(ab)\stackrel{\eqref{eq:delta}}=
(\id \ox \varepsilon)(\piR\ox \piR)T_2(a\ox b)=
\piR(\id \ox \epsilon)T_2(a\ox b)\stackrel{(iii)}=
\piR(ab).
$$
In order to prove counitality on the other side, we need an alternative
expression of $\delta$. To this end, note that for any $a,b,c\in A$,
\begin{equation}\label{eq:b_piR(a)c}
(\id \ox \epsilon)[(1\ox a)T_2(b\ox c)]=
(\id \ox \epsilon)[(b\ox 1)T_3(c\ox a)]
\stackrel{\eqref{eq:piR(a)b}}=
b\piR(a)c.
\end{equation}
On the other hand, for any $a,b,c\in A$,
\begin{eqnarray*}
(\id\ox \epsilon)[T_3(b\ox a)(1\ox c)]&=&
(\id\ox \epsilon)[(1\ox a)T_1(b\ox c)]\\
&=&(\id\ox \epsilon)[(1\ox \pibarL(a))T_1(b\ox c)]\\
&=&(\id\ox \epsilon)[(\piR(a)\ox 1)T_1(b\ox c)].
\end{eqnarray*}
The second equality follows by Lemma \ref{lem:counit_exp} and the third one
follows by Lemma \ref{lem:E_balanced}. Therefore,
\begin{eqnarray}\label{eq:T_3>2}
(\piR\ox \epsilon)\!\!\!\!\!\!\!&&\!\!\!\!\!\!\![T_3(b\ox a)(1\ox c)]=
(\piR\ox \epsilon)[(\piR(a)\ox 1)T_1(b\ox c)]\\
&=&(\piR\ox \epsilon)[(a\ox 1)T_1(b\ox c)]=
(\piR\ox \epsilon)[T_2(a\ox b)(1\ox c)],\nonumber
\end{eqnarray}
where the second equality follows by \eqref{eq:cond_exp_reg}. With these
identities at hand, for any $a,b,c,d,f,g\in A$,
\begin{eqnarray}\label{eq:3>2}
(f\ox g)\!\!\!\!\!\!\!\!\!\!&&\!\!\!\!\!\!\!\!\!\!
((\piR\ox \piR)T_3(b\ox a))(c\ox d)\\
&\stackrel{\eqref{eq:b_piR(a)c}}=&
(f\ox 1)(\piR\ox \epsilon \ox \id)[(T_3(b\ox a)\ox 1)(1\ox T_2^\op(g\ox d))]
(c\ox 1)\nonumber\\
&\stackrel{\eqref{eq:T_3>2}}=&
(f\ox 1)(\piR\ox \epsilon \ox \id)[(T_2(a\ox b)\ox 1)(1\ox T_2^\op(g\ox d))]
(c\ox 1)\nonumber\\
&\stackrel{\eqref{eq:b_piR(a)c}}=&
(f\ox g)((\piR\ox \piR)T_2(a\ox b))(c\ox d), \nonumber
\end{eqnarray}
so that by \eqref{eq:delta}, 
\begin{equation}\label{eq:delta_other}
\delta\piR(ab)=(\piR\ox \piR)T_3(b\ox a), \qquad \forall a,b\in A.
\end{equation}
Using this expression of $\delta$,
$$
(\varepsilon\ox \id)\delta \piR(ab)\stackrel{\eqref{eq:delta_other}}=
(\varepsilon\ox \id)(\piR\ox \piR)T_3(b\ox a)=
\piR(\epsilon \ox \id)T_3(b\ox a)\stackrel{(iii)}=
\piR(ab).
$$
\end{proof}

\begin{lemma}\label{lem:base_coalg_bar}
Let $A$ be a regular weak multiplier bialgebra over a field with a right full
comultiplication. Then the comultiplication $\delta$ and the counit
$\varepsilon$ in Theorem \ref{thm:base_coalg} satisfy the following
identities, for all $a,b\in A$.
\begin{itemize}
\item[{(1)}] $\delta\pibarR(ab)=(\pibarR\ox \pibarR)T_4^\op(b\ox a)$.
\item[{(2)}] $\delta\pibarR(ab)=(\pibarR\ox \pibarR)T_1^\op(a\ox b)$.
\item[{(3)}] $\varepsilon \pibarR(a)=\epsilon(a)$.
\end{itemize}
\end{lemma}

\begin{proof}
(1) Symmetrically to the derivation of \eqref{eq:delta}, for any $a,b,c,d,f\in
A$ denote $T_4(b\ox a)=:b'\ox a'$ and $T_2(c\ox d)=:c'\ox d'$, allowing for
implicit summation. Then
\begin{eqnarray*}
(cd\ox f)\delta\pibarR(ab)&=&
(cd\ox f)F(1\ox \pibarR(ab))=
c'\ox f\piR(d')\pibarR(ab)\\
&=& c'\ox f\pibarR(\piR(d')ab)=
cd\pibarR(a')\ox f\pibarR(b')\\
&=&(cd\ox f)((\pibarR \ox \pibarR)T_4^\op(b\ox a)).
\end{eqnarray*}
The third equality follows by Lemma \ref{lem:bar_mod_map} and the fourth
equality follows by the first assertion in Lemma \ref{lem:pi_(ef)}.

(2). Applying the equality \eqref{eq:3>2} in the opposite of the algebra $A$,
we obtain 
$$
(\pibarR \ox \pibarR)T_4(b\ox a)=(\pibarR \ox \pibarR)T_1(a\ox b).
$$
Hence the claim follows by part (1).

(3). Applying Proposition \ref{prop:base_counit} to the opposite of the
algebra $A$, there is a linear map $\pibarR(A)\to k$, $\pibarR(a)\mapsto
\epsilon(a)$. Using parts (1) and (2), it can be seen to be the counit for
$\delta$ proving that it is equal to $\varepsilon$.
\end{proof}

The following theorem describes the rich algebraic structure carried by the
base algebras. Such a result was obtained for {\em regular} weak multiplier
Hopf algebras in \cite{VDae:Sep}.

\begin{theorem}\label{thm:base_sF}
Let $A$ be a regular weak multiplier bialgebra over a field with a right
full comultiplication. Then the following assertions hold.
\begin{itemize}
\item[{(1)}] Via the coalgebra structure in Theorem \ref{thm:base_coalg} and
the restriction of the multiplication in $\M(A)$, $\piR(A)$ is a coseparable
coalgebra, hence a firm Frobenius algebra.
\item[{(2)}] The multiplication in the firm Frobenius algebra in part (1) is
non-degenerate. Moreover, it has (idempotent) local units. 
\item[{(3)}] The coalgebra $\piR(A)$ in part (1) is a co-Frobenius coalgebra.
Hence there exists a unique isomorphism of non-unital algebras 
$\vartheta: \piR(A)\to \piR(A)$ --- known as the {\em Nakayama automorphism}
--- such that $\varepsilon(sr)=\varepsilon (\vartheta(r)s)$, for all
$s,r\in \piR(A)$. 
\end{itemize}
\end{theorem}

\begin{proof}
(1). By Theorem \ref{thm:base_coalg}, $\piR(A)$ is a coalgebra. By Proposition
\ref{prop:F}~(3), the multiplication in $\piR(A)$ is a bicomodule retraction
(i.e. left inverse) of the comultiplication. This precisely means a
coseparable coalgebra structure. Then $\piR(A)$ is a firm Frobenius algebra
by the considerations in \cite[Section 6.4]{BoGT:fF}.

(2). For some $a\in A$, assume that $\piR(a)\pibarR(b)= \piR(a\pibarR(b))=0$,
for all $b\in A$ (where the first equality follows by Lemma
\ref{lem:bar_mod_map}). Then also 
$$
0=\varepsilon\piR(a\pibarR(b))=\epsilon(a\pibarR(b))=\epsilon(ab),
\quad \forall b\in A,
$$
where the last equality follows by Lemma \ref{lem:counit_exp}. This implies
that 
$$
0=(\id \ox \epsilon)[(1\ox a)T_2(b\ox c)]=
(\id \ox \epsilon)[(b\ox 1)T_3(c\ox a)]\stackrel{\eqref{eq:piR(a)b}}=
b\piR(a)c,\quad \forall b,c\in A,
$$
proving $\piR(a)=0$. Since $\piR(A)=\pibarR(A)$ by Theorem \ref{thm:full},
this proves the non-degeneracy of the multiplication on the
right. Non-degeneracy on the left is proven symmetrically. The existence of
local units follows by \cite[Proposition 7]{BoGT:fF}.

(3). In light of part (2), it follows by \cite[Proposition 7]{BoGT:fF} that
the coalgebra $\piR(A)$ in Theorem \ref{thm:base_coalg} is left and right
co-Frobenius. So the existence of the Nakayama automorphism follows by
\cite[Section 6]{CaDaNa}. 
\end{proof}

The following symmetric version is immediate.

\begin{theorem}\label{thm:base_sF_left}
Let $A$ be a regular weak multiplier bialgebra over a field $k$ with a left
full comultiplication. Then the following assertions hold.
\begin{itemize}
\item[{(1)}] The subalgebra $\piL(A)$ of $\M(A)$ is a coseparable
coalgebra, hence a firm Frobenius algebra.
\item[{(2)}] The multiplication in the firm Frobenius algebra in part (1) is
non-degenerate. Moreover, it has (idempotent) local units.
\item[{(3)}] The coalgebra $\piL(A)$ in part (1) is a co-Frobenius coalgebra
 (hence its counit has a Nakayama automorphism). 
\end{itemize}
\end{theorem}

Our next aim is to find a more explicit expression of the Nakayama
automorphisms in Theorem \ref{thm:base_sF}~(3) and Theorem
\ref{thm:base_sF_left}~(3). 

\begin{lemma}\label{lem:sigma}
For a regular weak multiplier bialgebra $A$ over a field, the following
assertions hold.
\begin{itemize} 
\item[{(1)}] If the comultiplication is left full, then there is a linear
anti-multiplicative map   
$$
\sigma:\piL(A)=\pibarL(A)\to \piR(A),\qquad 
\pibarL(a)\mapsto \piR(a).
$$ 
\item[{(2)}] If the comultiplication is left full, then there is a linear
anti-multiplicative map
$$
\overline\sigma:\piL(A)=\pibarL(A)\to \pibarR(A),\qquad 
\piL(a)\mapsto \pibarR(a).
$$ 
\item[{(3)}] If the comultiplication is right full, then there is a linear
anti-multiplicative map  
$$
\tau:\piR(A)=\pibarR(A)\to \pibarL(A),\qquad 
\piR(a)\mapsto \pibarL(a).
$$ 
\item[{(4)}] If the comultiplication is right full, then there is a linear
anti-multiplicative map 
$$
\overline\tau:\piR(A)=\pibarR(A)\to \piL(A),\qquad 
\pibarR(a)\mapsto \piL(a).
$$
\end{itemize}
If the comultiplication is both left and right full, then $\tau=\sigma^{-1}$
and $\overline\tau=\overline\sigma^{\,-1}$.
\end{lemma}

\begin{proof}
We prove part (1), all other parts follow symmetrically. Denote the counit of
the coalgebra $\piL(A)=\pibarL(A)$ by $\varepsilon$. If $\pibarL(a)=0$, then
for all $b,c\in A$,
\begin{eqnarray*}
0&=&
(\id \ox \varepsilon(\pibarL(a)\pibarL(-)))T_2(b\ox c)=
(\id \ox \varepsilon\pibarL(a\pibarL(-)))T_2(b\ox c)\\
&=&(\id \ox \epsilon(a\pibarL(-)))T_2(b\ox c)=
(\id \ox \epsilon(-\piR(a)))T_2(b\ox c)\\
&=&(\id \ox \epsilon)[(b\ox 1)\Delta(c)(1\ox \piR(a))]=
(\id \ox \epsilon)T_2(b\ox c\piR(a))\stackrel{(iii)}=
bc\piR(a),
\end{eqnarray*}
proving that $\piR(a)=0$. The second equality follows by Lemma
\ref{lem:mod_map} and the fourth one follows by Lemma \ref{lem:counit_inv}. In
the penultimate equality we applied axiom (iv) in Definition \ref{def:mwba},
\eqref{eq:coproduct_reg} and the multiplicativity of $\overline \Delta$. 
This proves the existence of the stated linear map $\sigma$. Using Lemma
\ref{lem:mod_map} in the first equality and Lemma \ref{lem:piRonL} in the
penultimate equality,
$$
\sigma(\pibarL(a)\pibarL(b))=
\sigma\pibarL(a\pibarL(b))=
\piR(a\pibarL(b))=
\piR(b)\piR(a)=
(\sigma\pibarL(b))(\sigma\pibarL(a)),
$$
for any $a,b\in A$; that is, $\sigma$ is anti-multiplicative.
\end{proof}

\begin{proposition}\label{prop:sigma}
Let $A$ be a regular weak multiplier bialgebra over a field with a left and
right full comultiplication. Then the maps $\sigma$ and $\overline\sigma$ in
Lemma \ref{lem:sigma} are anti-coalgebra isomorphisms. Moreover, the Nakayama
automorphism of $\piR(A)$ is equal to $\sigma \overline \sigma^{-1}$ and the
Nakayama automorphism of $\piL(A)$ is equal to $\overline \sigma^{-1} \sigma$.
\end{proposition}

\begin{proof}
By the left-right symmetric counterpart of Lemma \ref{lem:base_coalg_bar}~(1),
for any $a,b\in A$, $\delta \pibarL(ab)=(\pibarL\ox \pibarL)T_3^\op(b\ox 
a)$. Therefore, 
$$
(\sigma\ox \sigma)\delta^\op \pibarL(ab)=
(\sigma\pibarL\ox \sigma\pibarL)T_3(b\ox a)=
(\piR\ox \piR)T_3(b\ox a)\stackrel{\eqref{eq:delta_other}}=
\delta\piR(ab),
$$
so that $\sigma$ is anti-comultiplicative. By the left-right symmetric
counterpart of Lemma \ref{lem:base_coalg_bar}~(3), $\varepsilon\sigma 
\pibarL(a)=\varepsilon \piR(a)=\epsilon(a)=\varepsilon \pibarL(a)$ for any
$a\in A$, proving that $\sigma$ is an anti-coalgebra map. It is proven by
symmetric steps that also $\overline \sigma$ is anti-multiplicative and an
anti-coalgebra homomorphism.

Applying Lemma \ref{lem:counit_exp} in the second equality, it follows for all
$a,b\in A$ that 
$$
\epsilon(a\overline\sigma\piL(b))=
\epsilon(a\pibarR(b))=
\epsilon(ab)\stackrel{\eqref{eq:counit_exp_reg}}=
\epsilon(a\piL(b)).
$$
Since $\piL(A)=\pibarL(A)$ by Theorem \ref{thm:full}, this implies that
$\epsilon(a\overline\sigma\pibarL(b))= \epsilon(a\pibarL(b))$, for all
$a,b\in A$. Using this identity in the fourth equality and Lemma
\ref{lem:counit_inv} in the fifth one, 
\begin{eqnarray*}
\varepsilon(\piR(a)\overline\sigma
\!\!\!\!\!\!\!&&\!\!\!\!\!\!\!\sigma^{-1}\piR(b))=
\varepsilon(\piR(a)\overline\sigma\pibarL(b))\stackrel{\eqref{eq:mod_map_reg}}=
\varepsilon\piR(a\overline\sigma\pibarL(b))=
\epsilon(a\overline\sigma\pibarL(b))\\
&=& \epsilon(a\pibarL(b))=
\epsilon(b\piR(a))=
\varepsilon\piR(b\piR(a))\stackrel{\eqref{eq:mod_map_reg}}=
\varepsilon(\piR(b)\piR(a)),
\end{eqnarray*}
for all $a,b\in A$. This proves that $\sigma \overline\sigma^{\,-1}$ is the
Nakayama automorphism of $\piR(A)$ and symmetric considerations prove that
$\overline \sigma^{\,-1} \sigma$ is the Nakayama automorphism of $\piL(A)$.
\end{proof}

Finally, we want to find a relation between the multipliers $E$ and $F$. 

\begin{lemma}\label{lem:E_restricted}
Let $A$ be a regular weak multiplier bialgebra over a field. Then for all
$a\in A$, 
$$
(1\ox \pibarL(a))E\in \piR(A)\ox \pibarL(A)\quad \textrm{and}\quad
E(1\ox \piL(a))\in \pibarR(A)\ox \piL(A).
$$
In particular, if the comultiplication is right and left full, then we can
regard $E$ as an element of $\M(\piR(A)\ox \piL(A)^\op)$.
\end{lemma}

\begin{proof}
For any $c,d,f\in A$, $(f\ox 1)T_4(d\ox c)=(f\ox 1)\Delta(c)(d\ox 1)=T_2(f\ox
c)(d\ox 1)$. Hence multiplying on the left both sides of
\eqref{eq:pi_identity_reg} by $f\ox 1$ and simplifying on the right the
resulting equality by $d\ox 1$, we obtain the identity 
$$
(f\ox 1)((\piR\ox \id)T_3(a\ox b))(c\ox 1)=
(1\ox ba)((\id \ox \piL)T_2(f\ox c)),
$$
for all $a,b,c,f\in A$. Using this identity in the fourth equality and Lemma
\ref{lem:bar_mod_map} in the third one,
\begin{eqnarray*}
(a\ox 1)(1\ox \pibarL(bc))
\!\!\!\!\!\!\!&\!\!\!\!\!\!\!\!E\!\!\!\!\!\!\!\!&\!\!\!\!\!\!\!(d\ox f)=
(1\ox \pibarL(bc))(a\ox 1)E(d\ox f)\\
&\stackrel{\eqref{eq:E(a@1)}}=&
(1\ox \pibarL(bc))((\id \ox \piL)T_2(a\ox d))(1\ox f)\\
&=&(1\ox \pibarL)[(1 \ox bc)((\id \ox \piL)T_2(a\ox d))](1\ox f)\\
&=&(1\ox \pibarL)[(a\ox 1)((\piR\ox \id)T_3(c\ox b))(d\ox 1)](1\ox f)\\
&=&(a\ox 1)((\piR\ox \pibarL)T_3(c\ox b))(d\ox f),
\end{eqnarray*}
for all $a,b,c,d,f\in A$. This proves
\begin{equation}\label{eq:E_restricted}
(1\ox \pibarL(bc))E=(\piR\ox \pibarL)T_3(c\ox b)\qquad \forall b,c\in A,
\end{equation}
hence by the idempotency of $A$ also the first claim. The second claim is
proven symmetrically. 
\end{proof}

Let $A$ be a regular weak multiplier bialgebra over a field with a right and
left full comultiplication. The algebra $\piR(A)\ox \piL(A)$ is idempotent by
Proposition \ref{prop:F}~(3) and its symmetric counterpart. Hence
the multiplicative and bijective map $\id\ox \sigma:\piR(A)\ox \piL(A)^\op\to
\piR(A)\ox \piR(A)$ in Lemma \ref{lem:sigma}~(1) is non-degenerate and thus
extends to an algebra homomorphism $\overline{\id \ox \sigma}: \M(\piR(A)\ox
\piL(A)^\op)\to \M(\piR(A)\ox \piR(A))$.

\begin{proposition}\label{prop:E<>F}
Let $A$ be a regular weak multiplier bialgebra over a field with a right and
left full comultiplication. Then $(\overline{\id \ox \sigma})(E)=F$ as
elements of $\M(\piR(A)\ox \piR(A))$. 
\end{proposition}

\begin{proof}
For any $a,b,c\in A$, 
\begin{eqnarray*}
((\overline{\id \ox \sigma})(E))(\piR(a)\ox \piR(bc))&=&
(\overline{\id \ox \sigma})[(1\ox \pibarL(bc))E(\piR(a)\ox 1)]\\
&\stackrel{\eqref{eq:E_restricted}}=&
(\overline{\id \ox \sigma})[((\piR\ox \pibarL)T_3(c\ox b))(\piR(a)\ox 1)]\\
&=&((\piR\ox \piR)T_3(c\ox b))(\piR(a)\ox 1)\\
&\stackrel{\eqref{eq:delta_other}}=&
F(\piR(a) \ox \piR(bc)),
\end{eqnarray*}
where in the first and the third equalities we used part (1) of Lemma
\ref{lem:sigma} and that $\overline{\id \ox \sigma}$ is multiplicative.
\end{proof}

\section{A monoidal category of modules}\label{sec:modcat}

Bialgebras over a field can be characterized by the property that the category
of their (left or right) modules is monoidal such that the forgetful functor
to the category of vector spaces is strict monoidal. More generally, the
category of (left or right) modules over a weak bialgebra is monoidal such
that the forgetful functor to the category of bimodules over the (separable
Frobenius) base algebra is strict monoidal (see e.g. \cite{Szl}). The aim of
this section is to prove a similar property of regular weak multiplier
bialgebras with a (left or right) full comultiplication. The key point in
doing so is to find the appropriate notion of module in the absence of an
algebraic unit.

\begin{definition}\label{def:module}
Let $A$ be an idempotent algebra over a field with a non-degenerate
multiplication. By a {\em non-unital} right $A$-module we mean a vector space
$V$ equipped with a linear map (called the {\em $A$-action}) $V\ox A \to V$,
$v\ox a\mapsto va$ satisfying the associativity condition 
$$
(va)b=v(ab)\qquad \forall v\in V,\ a,b\in A.
$$
A right $A$-module $V$ is said to be {\em idempotent} (or {\em unital}) if
the $A$-action $V\ox A\to V$ is surjective. It is called {\em firm} if the
quotient map $V\ox_A A \to V$, $v\ox_A a\mapsto va$ (to the $A$-module tensor
product $V\ox_A A$) is bijective. Finally, $V$ is a {\em non-degenerate}
$A$-module if for any $v\in V$, $(va=0\ \forall a\in A)$ implies $v=0$. Left 
$A$-modules are defined as right modules over the opposite algebra $A^\op$,
with action denoted by $V\ox A^{\mathsf{op}}\to V$, $v\ox a \mapsto av$;
and $A$-bimodules are both left and right $A$-modules $V$ with
commuting actions (i.e. such that $a(vb)=(av)b$, for all $v\in V$ and $a,b\in
A$). 
\end{definition}

The morphisms of non-unital modules over an idempotent and non-degenerate
algebra $A$ are the linear maps $f:V\to W$ such that $f(va)=f(v)a$, for all
$v\in V$ and $a\in A$. Throughout, we denote by $M_{(A)}$ the category of
idempotent and non-degenerate right $A$-modules. The category of firm
$A$-bimodules (i.e. of bimodules which are firm both as left and right
modules) will be denoted by ${}_A M_A$. Whenever $A$ is a firm algebra 
--- that is, the quotient map $A\ox_A A\to A$, $a\ox_A b\mapsto ab$ is
bijective --- ${}_A M_A$ is a monoidal category via the module tensor product
$\ox_A$ and the neutral object $A$. 

Let $A$ be a regular weak multiplier bialgebra over a field with a right full
comultiplication. By Theorem \ref{thm:base_sF}~(1), $R:=\piR(A)$ is a firm
algebra so there is a monoidal category ${}_R M_R$. 

\begin{proposition}\label{prop:forgetful}
Let $A$ be a regular weak multiplier bialgebra over a field with a right full
comultiplication. Any object of $M_{(A)}$ can be regarded as a firm
$R:=\piR(A)$-bimodule. This gives rise to a functor $U:M_{(A)} \to {}_R M_R$,
acting on the morphisms as the identity map.
\end{proposition}

\begin{proof} 
Using that any object $V$ of $M_{(A)}$ is an idempotent $A$-module, define the
$R$-actions on $V$ with the help of the map $\tau$ in Lemma
\ref{lem:sigma}~(3) by
$$
(va)\cdot \piR(b):=v(a\piR(b))\quad \textrm{and}\quad 
\piR(b)\cdot (va):=v(a(\tau \piR(b)))=v(a\pibarL(b)).
$$
In order to see that these actions are well-defined, assume that $va=0$. Then
for all $b,c\in A$,
$$
0=(va)(\piR(b)c)=v(a(\piR(b)c))=v((a\piR(b))c)=(v(a\piR(b)))c.
$$
So by the non-degeneracy of $V$, $0=v(a\piR(b))$ proving that the right
$R$-action on $V$ is well-defined. One checks symmetrically that also the left
$R$-action is well-defined. Associativity of both actions is evident by the
associativity of the multiplication in $\M(A)$ and the
anti-multiplicativity of $\tau$ (cf. Lemma \ref{lem:sigma}~(3)). The 
left and right $R$-actions commute by Lemma \ref{lem:commute} (since by the
right fullness of the comultiplication $\piR(A)=\pibarR(A)$, see
Theorem \ref{thm:full}). Finally, $R$ has local units by Theorem
\ref{thm:base_sF}~(2). So in order to see that $V$ is a firm $R$-bimodule, it
is enough to see that it is idempotent as a left and as a right
$R$-module. Since both the algebra $A$ and the module $V$ are idempotent, any
element of $V$ can be written as a linear combination of elements of the form
$v(ab)=v(a'\piR(b'))=(va')\cdot \piR(b')$, in terms of $v\in V$ and $a,b\in
A$, where $a'\ox b':=T_2(a\ox b)$ (allowing for implicit summation) and the
first equality follows by Lemma \ref{lem:counital_maps}~(4). Symmetrically,
any element of $V$ can be written as a linear combination of elements of the
form $w(cd)=w (c'\pibarL(d'))=\piR(d')\cdot (wc')$, in terms of $w\in V$ and
$c,d\in A$, where $d'\ox c':= T_3(d\ox c)$ (allowing for implicit summation)
and the first equality follows by Lemma \ref{lem:counital_maps}~(1). 

With respect to the stated $R$-actions, any $A$-module map is evidently a
morphism of $R$-bimodules. This proves the existence of the stated functor
$U$. 
\end{proof}

\begin{proposition}\label{prop:A_module R}
Let $A$ be a regular weak multiplier bialgebra over a field with a right full
comultiplication. Then $R:=\piR(A)$ carries the structure of an idempotent and
non-degenerate right $A$-module. The functor $U$ in Proposition
\ref{prop:forgetful} takes this object $R$ of $M_{(A)}$ to the $R$-bimodule
$R$ with the actions provided by the multiplication.  
\end{proposition}

\begin{proof}
For any $a,b\in A$, put
$$
\piR(a)\wr b:=
\piR(\piR(a)b)\stackrel{\eqref{eq:cond_exp_reg}}=\piR(ab). 
$$
It is clearly a well-defined associative action. Let us see that it is
idempotent. For any $a,b\in A$, denote $b'\ox a':=T_4(b\ox a)$, allowing for
implicit summation. By the right fullness of the comultiplication,
$\piR(A)=\pibarR(A)$, cf. Theorem \ref{thm:full}. So by Lemma
\ref{lem:counital_maps}~(2), 
$$
\pibarR(a')\wr b'=\piR(\pibarR(a')b')=\piR(ab).
$$
By the idempotency of $A$, this proves that the $A$-action on $R$ is
surjective. In order to see its non-degeneracy, assume that for some $a\in A$,
$\piR(ab)=0$ for all $b\in A$. Then for all $b,c,d\in A$,
\begin{eqnarray*}
0&=&
(\mu(\id \ox \piR)[(1\ox a)T_2(b\ox c)])d\\
&=&(\mu(\id \ox \piR)[(b\ox 1)T_3(c\ox a)])d\\
&\stackrel{\eqref{eq:piR(a)b}}=&
b(\mu(\id\ox \id \ox \epsilon)(\id \ox T_3\mathsf{tw})
(T_3\ox \id)(c\ox a\ox d))\\
&=&b((\id \ox \epsilon)(\mu\ox \id)(\id \ox T_3\mathsf{tw})
(T_3\ox \id)(c\ox a\ox d))\\
&=&b((\id \ox \epsilon)T_3(cd\ox a))\stackrel{\eqref{eq:piR(a)b}}=
b\piR(a)cd.
\end{eqnarray*} 
By the density of $A$ in $\M(A)$, this proves $\piR(a)=0$
hence the non-degeneracy of the action. In the penultimate equality we used
that 
$$
(\mu\ox \id)(\id \ox T_3\mathsf{tw})(T_3\ox \id)(c\ox a\ox d)=
(1\ox a)\Delta(c)\Delta(d)=
(1\ox a)\Delta(cd)=
T_3(cd\ox a).
$$
Applying the functor $U:M_{(A)}\to {}_R M_R$ in Proposition
\ref{prop:forgetful} to the object $R$ of $M_{(A)}$ above, the right action in
the resulting $R$-bimodule comes out as the right multiplication. Indeed,
$$
\piR(ab)\cdot \piR(c)=
\piR(a)\wr (b\piR(c))=
\piR(ab\piR(c))\stackrel{\eqref{eq:mod_map_reg}}=
\piR(ab)\piR(c),
$$
for all $a,b,c\in A$. The left $R$-action is also given by the multiplication
since 
$$
\piR(c)\cdot \piR(ab)=
\piR(a)\wr (b\pibarL(c))=
\piR(ab\pibarL(c))=
\piR(c)\piR(ab),
$$
where the last equality follows by Lemma \ref{lem:piRonL}.
\end{proof}

\begin{lemma}\label{lem:R_mod_prod}
Let $A$ be a regular weak multiplier bialgebra over a field with a right full
comultiplication. Regard any objects $V$ and $W$ of $M_{(A)}$ as firm
$R:=\piR(A)$-bimodules as in Proposition \ref{prop:forgetful}. Then the
$R$-module tensor product $V\ox_R W$ is isomorphic to 
$$
\langle (v\ox w)((a\ox b)E)\ |\ v\in V,\ w\in W,\ a,b\in A \rangle.
$$
\end{lemma}

\begin{proof}
By Theorem \ref{thm:base_sF}~(1), $R$ is a coseparable coalgebra. Then by
\cite[Proposition 2.17]{BoVe}, $V\ox_R W$ is isomorphic to the image of the
idempotent map $\theta:V\ox W \to V\ox W$,
$$
v\cdot \piR(a)\ox w\mapsto 
(v\cdot(-)\ox (-)\cdot w) \delta\piR(a),
$$
where $\delta:R\to R\ox R$ is the ($R$-bilinear) comultiplication in
Proposition \ref{prop:F}~(3). In order to obtain a more explicit expression
of this map $\theta$, note that by the idempotency of $A$ and Lemma
\ref{lem:counital_maps}~(4), any element of $A$ can be written as a linear
combination of elements of the form $a\piR(bc)$ --- so any element of $V$ can
be written as a linear combination of elements of the form 
$va\piR(bc)$ --- in terms of $v\in V$, $a,b,c\in A$. Now 
\begin{eqnarray*}
\theta(va\piR(bc)\ox wd)&\stackrel{\eqref{eq:delta}}=&
(va\ox wd)((\piR\ox \pibarL)T_2(b\ox c))\\
&=&(va\ox wd)((\piR\ox \id)[(bc\ox 1)E])\\
&=&(v\ox w)((a\piR(bc)\ox d)E),
\end{eqnarray*}
where the second equality follows by Proposition \ref{prop:(a@1)E}~(1) and in
the last equality we used that for all $a,b,c,d\in A$, 
\begin{eqnarray*}
(\piR\ox \id)[(ab\ox cd)E]&\stackrel{\eqref{eq:E(a@1)}}=&
(\piR\ox \id)[(ab\ox 1)((\piR\ox \id)T_3(d\ox c))]\\
&\stackrel{\eqref{eq:mod_map_reg}}=&
(\piR(ab)\ox 1)((\piR\ox \id)T_3(d\ox c))
\stackrel{\eqref{eq:E(a@1)}}=
(\piR(ab)\ox cd)E,
\end{eqnarray*}
hence $(\piR\ox \id)[(ab\ox 1)E]=(\piR(ab)\ox 1)E$. This proves that the image
of $\theta$ is spanned by the stated elements $(v\ox w)((a\ox b)E)$. 
\end{proof}

\begin{proposition}\label{prop:product_A_module}
Let $A$ be a regular weak multiplier bialgebra over a field with a right full
comultiplication. Regard any objects $V$ and $W$ of $M_{(A)}$ as firm
$R:=\piR(A)$-bimodules as in Proposition \ref{prop:forgetful}. Then the
$R$-module tensor product $V\ox_R W$ carries the structure of an idempotent
and non-degenerate $A$-module too.
\end{proposition}

\begin{proof}
Observe that, for each $c \in A$, there is a well-defined linear map $V \ox W
\to V \ox W$ given by 
$$
va \ox wb \mapsto (va \ox w)T_3(c\ox b)=
(v \ox wb)T_2(a\ox c)=
(v \ox w)((a \ox b) \Delta(c))
$$
for any $a,b \in A$, $v\in V$ and $w\in W$. Since this map is
$R$-balanced by Lemma \ref{lem:E_balanced}, it induces an action of $A$
on $V \ox_R W$ defined by 
$$
(va\ox_R wb) c=\pi((v\ox w)((a\ox b)\Delta(c))),
$$
where $\pi:V\ox W \to V\ox_R W$ denotes the canonical epimorphism. This action
is associative by the multiplicativity of $\Delta$. In order to see that
it is idempotent and non-degenerate, let us apply the isomorphism in
Lemma \ref{lem:R_mod_prod}. It takes the above $A$-action on $V\ox_R W$ to 
\begin{equation}\label{eq:product_A_module}
(v\ox w)((a\ox b)E) c=
(v\ox w)((a\ox b)\Delta(c)).
\end{equation}
It is an idempotent action by axiom (iv) in Definition \ref{def:mwba}. In
order to see that it is non-degenerate, assume that $(v\ox w)((a\ox
b)\Delta(c))=0$ for all $c\in A$. Then 
$$
\begin{array}{lll}
0=(v\ox w)((a\ox b)\Delta(c))(d\ox f)=
(v\ox w)(a\ox b)(\Delta(c)(d\ox f))
\ &\forall c,d,f\in A
&\stackrel{(iv)}\Longrightarrow\\
0=(v\ox w)(a\ox b)(E(c\ox d))=
(v\ox w)((a\ox b)E)(c\ox d)\ &\forall c,d\in A.&
\end{array}
$$
By \cite[Lemma 1.11]{JaVe}, $V\ox W$ is a non-degenerate $A\ox
A$-module. Hence $0=(v\ox w)((a\ox b)E)$, proving the non-degeneracy of the
$A$-module $V\ox_R W$. 

Applying the functor $U:M_{(A)}\to {}_R M_R$ in Proposition
\ref{prop:forgetful} to the object $V\ox_R W$ of $M_{(A)}$ above, it follows
by Lemma \ref{lem:coproduct} and \eqref{eq:coproduct_reg} that the resulting
$R$-bimodule has the actions 
$$
\piR(a)\cdot (v\ox_R w)\cdot \piR(b)=(\piR(a)\cdot v)\ox_R (w\cdot \piR(b)).
$$
\end{proof}

\begin{theorem}\label{thm:mod_cat}
Let $A$ be a regular weak multiplier bialgebra over a field with a right full
comultiplication. Then $M_{(A)}$ is a monoidal category and the functor
$U:M_{(A)}\to {}_R M_R$ in Proposition \ref{prop:forgetful} is strict monoidal. 
\end{theorem}

\begin{proof}
In view of Proposition \ref{prop:A_module R} and Proposition
\ref{prop:product_A_module}, we only need to show that the associativity and
unit constraints of ${}_R M_R$ --- if evaluated on objects of $M_{(A)}$ ---
are morphisms of $A$-modules.

Take any objects $V,W,Z$ in $M_{(A)}$. In view of Lemma \ref{lem:R_mod_prod},
$(V\ox_R W)\ox_R Z$ is isomorphic to the vector subspace of $V\ox W\ox Z$
spanned by the elements of the form
\begin{eqnarray*}
((v\ox w)((a\ox b)E)\ox z)((c\ox d)E)&\stackrel{\eqref{eq:product_A_module}}=&
(v\ox w\ox z)((a\ox b\ox 1)(\Delta\ox \id)((c\ox d)E))\\
&\stackrel{(v) (iv) }=&
(v\ox w\ox z)((a\ox b\ox 1)(\Delta(c)\ox d)(1\ox E)),
\end{eqnarray*}
(for $a,b,c,d\in A$, $v\in V$, $w\in W$ and $z\in Z$) hence in light of axiom
(iv) in Definition \ref{def:mwba}, by elements of the form
$$
(v\ox w\ox z)((a\ox b\ox d)(E\ox 1)(1\ox E))\stackrel{(v)}=
(v\ox w\ox z)((a\ox b\ox d)(1\ox E)(E\ox 1))
$$
(for $a,b,d\in A$, $v\in V$, $w\in W$ and $z\in Z$).
A symmetric computation shows the isomorphism of the same vector subspace of
$V\ox W\ox Z$ to $V\ox_R (W\ox_R Z)$, and the associator isomorphism $(V\ox_R
W)\ox_R Z\to V\ox_R (W\ox_R Z)$ is given by the composite of these
isomorphisms. Its $A$-module map property is thus equivalent to the equality of
both induced actions 
\begin{eqnarray}\label{eq:(VW)Z}
((v\ox w\ox z)\!\!\!\!\!\!\!\!&&\!\!\!\!\!\!\!\!
((a\ox b\ox c)(E\ox 1)(1\ox E))) d \\
&=&((v\ox w)((a\ox b)E)\ox z) T_3(d\ox c)\nonumber\\
&=&(v\ox wb\ox z)((T_2\ox \id)(\id \ox T_3)(a\ox d\ox c))
\nonumber
\end{eqnarray}
and 
\begin{eqnarray}\label{eq:V(WZ)}
((v\ox w\ox z)\!\!\!\!\!\!\!\!&&\!\!\!\!\!\!\!\!
((a\ox b\ox c)(1\ox E)(E\ox 1)) d \\
&=&(v\ox (w\ox z)((b\ox c)E)) T_2(a\ox d)\nonumber\\
&=&(v\ox wb\ox z)((\id \ox T_3)(T_2\ox \id)(a\ox d\ox c)),
\nonumber
\end{eqnarray}
where we used the equivalent forms $(v\ox w)((a\ox b)E)
c=(va\ox w)T_3(c\ox b)=(v\ox wb)T_2(a\ox c)$ of the action in
\eqref{eq:product_A_module}. The actions \eqref{eq:(VW)Z} and \eqref{eq:V(WZ)} 
are equal by the alternative form $(T_2\ox \id)(\id \ox T_3)=(\id \ox
T_3)(T_2\ox \id)$ of the coassociativity axiom (ii) in Definition
\ref{def:mwba} (obtained by evaluating both sides of (ii) on any $a\ox b\ox c
\in A\ox A \ox A$, multiplying on the left by $1\ox 1\ox d$ and simplifying on
the right by $1\ox 1\ox c$). 

In order to see that the left unit constraint $R\ox_R V \to V$,
$\piR(a)\ox_R v \mapsto v\pibarL(a)$ is a morphism of right $A$-modules, take
any $\piR(a)\ox_R vb \in R\ox_R V$. Applying to it the left unit constraint
and next the action by any $c\in A$ results in $vb\pibarL(a)c$. On the other
hand, acting first by $c\in A$ on $\piR(a)\ox_R vb$ yields $\pi((\piR(a)\ox
v) T_3(c\ox b))$, where we used the notation $\pi:R\ox V \to
R\ox_R V$ for the canonical epimorphism. Applying now the left unit
constraint, we obtain $v(\mu^\op(\pibarL\ox \id)[(a\ox 1)T_3(c\ox b)])$. Using
Lemma \ref{lem:cond_exp}, Lemma \ref{lem:coproduct} and Lemma
\ref{lem:counital_maps}~(1) in the first, second and last equalities, 
respectively, we see that for any $a,b,c\in A$
\begin{eqnarray*}
\mu^\op(\pibarL\ox \id)[(a\ox 1)T_3(c\ox b)]
&=&\mu^\op(\pibarL\ox \id)
[(\pibarL(a)\ox 1)T_3(c\ox b)]\\
&=&\mu^\op(\pibarL\ox \id)T_3(\pibarL(a)c \ox b)=
b\pibarL(a)c.
\end{eqnarray*}
This proves that the left unit constraint in ${}_R M_R$ evaluated on an object
$V$ of $M_{(A)}$ is a morphism of $A$-modules. A symmetric reasoning applies
for the right unit constraint. 
\end{proof}

\section{The antipode}\label{sec:antipode}

The {\em antipode} of a Hopf algebra $A$ is defined as the convolution inverse
of the identity map $A\to A$. In a {\em weak} Hopf algebra, the antipode is
no longer a strict inverse of the identity map in the convolution algebra of
the linear maps $A\to A$. However, it is a `weak' inverse in the following
sense. The maps $\piL$ and $\piR:A\to A$ are idempotent elements in the
convolution algebra and they serve as left, respectively, right units for the
identity map $A\to A$. The antipode is then a linear map $A\to A$ for whom
$\piL$ and $\piR$ serve as a right, respectively, left unit; and whose
convolution products with the identity map in both possible orders yield
$\piL$ and $\piR$, respectively. 
In what follows, we equip a weak multiplier bialgebra with an antipode in
the above spirit: as a generalized convolution inverse. The resulting
structure is compared with a (regular or not) weak multiplier
Hopf algebra in \cite{VDaWa}. 

Let $A$ be a regular weak multiplier bialgebra over a field. By
Proposition \ref{prop:F}~(1), for all $a,b\in A$ $(ab\ox 1)F=(\id\ox
\piR)T_2(a\ox b)$ is an element of $A\ox \M(A)$. So by the idempotency of $A$,
$(a\ox 1)F\in A\ox \M(A)$ for all $a\in A$, allowing for the definition of a
linear map 
\begin{equation}\label{eq:G1}
G_1:A\ox A \to A \ox A, \qquad a\ox b\mapsto (a\ox 1)F(1\ox b).
\end{equation}
The notation $G_1$ is motivated by the fact that it is the same map appearing
under the same name in \cite[Proposition 1.11]{VDaWa}:  

\begin{proposition}\label{prop:G1}
Let $A$ be a regular weak multiplier bialgebra. The map \eqref{eq:G1}
satisfies the equality 
$$
(G_1\ox \id)[\Delta_{13}(a)(1\ox b\ox c)]=
\Delta_{13}(a)(1\ox E)(1\ox b\ox c),\qquad \forall a,b,c\in A.
$$
Hence it is the same map denoted by $G_1$ in \cite[Proposition 1.11]{VDaWa}.
\end{proposition}

\begin{proof}
For any $a,b,c,d\in A$, 
\begin{eqnarray*}
(G_1\ox \id)[\Delta_{13}(a)(1\ox bd\ox c)]&=&
\Delta_{13}(a)((\pibarR\ox \id)T_4^\op(d\ox b)\ox c)\\
&=&\Delta_{13}(a)(1\ox (\id \ox \piL)T_4(d\ox b))(1\ox 1\ox c)\\
&=&\Delta_{13}(a)(1\ox E)(1\ox bd\ox c).
\end{eqnarray*}
The first equality follows by Proposition \ref{prop:F}~(1), the second one
follows by Lemma \ref{lem:E_balanced} and the last equality follows by
\eqref{eq:E(a@1)}. So we conclude by the idempotency of $A$.
\end{proof}

In \cite{VDaWa}, the form \eqref{eq:G1} of $G_1$ was proven for {\em regular}
weak multiplier Hopf algebras, but it was left open if it has the above form
for {\em arbitrary} weak multiplier Hopf algebras.

\begin{proposition}\label{Im<>mod_prod}
Let $A$ be a regular weak multiplier bialgebra over a field. If the
comultiplication is left and right full, then the following hold.
\begin{itemize}
\item[{(1)}] The image of the map $E_1:A\ox A \to A\ox A$, $a\ox b\mapsto
E(a\ox b)$ is isomorphic to the $\piL(A)$-module tensor square of $A$ with
respect to the actions 
$$
\piL(a)\cdot b:=\piL(a)b\qquad \textrm{and}\qquad
b\cdot \piL(a):=\pibarR(a)b,\quad \textrm{ for}\ a,b\in A.
$$
\item[{(2)}] The image of the map $G_1:A\ox A \to A\ox A$ in \eqref{eq:G1} is
isomorphic to the $\piR(A)$-module tensor square of $A$ with respect to the
actions 
$$
\piR(a)\cdot b:=\piR(a)b\qquad \textrm{and}\qquad
b\cdot \piR(a):=b\piR(a),\quad \textrm{ for}\ a,b\in A.
$$
\end{itemize}
\end{proposition}

\begin{proof}
(1). By a symmetric version of Lemma \ref{lem:R_mod_prod}, $E_1$ is equal to
the map
$$
a\ox \piL(bc)d\mapsto ((\pibarR\ox \piL)T_1(b\ox c))(a\ox d)
$$
whose image is equal to the stated module tensor product. 

(2). Since $\piR(A)$ is a coseparable coalgebra by Theorem
\ref{thm:base_sF}~(1), it follows by \cite[Proposition 2.17]{BoVe} that the
stated module tensor product is isomorphic to the image of the map
$$
a\piR(b)\ox c \mapsto 
(a\ox 1)(\delta \piR(b))(1\ox c)=
(a\piR(b)\ox 1)F(1\ox c)=
G_1(a\piR(b)\ox c),
$$
where $F\in \M(A\ox A)$ appeared in Proposition \ref{prop:F}~(1) and
$\delta:\piR(A)\to \piR(A)\ox \piR(A)$ is the comultiplication in 
Proposition \ref{prop:F}~(3). Since by Lemma \ref{lem:counital_maps}~(4) and
by the idempotency of $A$ any element of $A$ is a linear combination of
elements of the form $a\piR(b)$, for $a,b\in A$, we have the claim proven.
\end{proof}

For any weak multiplier bialgebra $A$ over a field, consider the vector space 
$$
\begin{array}{lll}
\mathcal L:=\{L:A\ox A \to A\ox A \ |& L(a\ox bc)=L(a\ox b)(1\ox c)\ 
\forall a,b,c\in A&\\
& (T_2\ox \id)(\id \ox L)=(\id \ox L)(T_2\ox \id)&\}.
\end{array}
$$
Denoting the vector space of right $A$-module maps $A\to A$ by
$\mathsf{End}_A(A)$, there is a linear map
\begin{equation}\label{eq:L_map}
\mathcal L\to \mathsf{Lin}(A,\mathsf{End}_A(A)),\qquad
L\mapsto [\lambda_L:a\mapsto (\epsilon\ox \id)L(a\ox -)].
\end{equation}
With its help, for any $a,b,c\in A$ and $L\in \mathcal L$,
\begin{eqnarray}\label{eq:L_form}
((\id \ox \lambda_L)\!\!\!&\!\!T_2\!\!&\!\!\!(a\ox b))(1\ox c)=
(\id \ox \epsilon \ox \id)(\id \ox L)(T_2 \ox \id)(a\ox b \ox c)\\
&&=(\id \ox \epsilon \ox \id)(T_2 \ox \id)(\id \ox L)(a\ox b \ox c)
\stackrel{(iii)}=
(a\ox 1)L(b\ox c).\nonumber
\end{eqnarray}
Applying this together with the non-degeneracy of the multiplication in $A\ox
A$, we conclude that the map \eqref{eq:L_map} is injective. Clearly, $\mathcal
L$ is an algebra via the composition of maps. For $L,L'\in \mathcal L$ and
$a\in A$, 
\begin{equation}\label{eq:L_conv}
\lambda_{L'L}(a)=
(\epsilon \ox \id)L'L(a\ox -)=
\mu(\lambda_{L'}\ox \id)L(a\ox -)
\end{equation}
 (where $\mu:\mathsf{End}_A(A)\ox A \to A$ denotes the evaluation map
$\Phi\ox a\mapsto \Phi a\equiv\Phi(a)$), generalizing the convolution
product $(\lambda_{L'}\ast \lambda_L)(a) = \mu(\lambda_{L'} \ox
\lambda_L)\Delta(a)$ of endomorphisms $\lambda_{L'}$ and $\lambda_L$ on a
(weak) bialgebra. 

\begin{proposition}\label{prop:L}
Let $A$ be a regular weak multiplier bialgebra over a field. 
\begin{itemize}
\item[{(1)}] The maps $T_1$, $E_1:=E(-\ox -)$ and $G_1$ in \eqref{eq:G1} from
$A\ox A$ to $A\ox A$ are elements of $\mathcal L$.
\item[{(2)}] The map \eqref{eq:L_map} takes the elements of $\mathcal L$ in
part (1) to $[a\mapsto a(-)]$, $[a\mapsto \piL(a)(-)]$ and $[a\mapsto
\piR(a)(-)]$, respectively. 
\item[{(3)}] $E_1^2=E_1$, $G_1^2=G_1$ and $E_1T_1=T_1=T_1G_1$.
\end{itemize}
\end{proposition}

\begin{proof}
(1). Evidently, all of $T_1$, $E_1$ and $G_1$ are right $A$-module maps. The
compatibility of $T_1$ with $T_2$ is axiom (ii) in Definition \ref{def:mwba}. 
The compatibility of $E_1$ with $T_2$ follows in the same way as in
\cite[Proposition 2.2]{VDaWa}: for all $a,b,c\in A$, 
\begin{eqnarray*}
(T_2\ox \id)(\id \ox E_1)(a\ox b\ox c)&=&
(a\ox 1\ox 1)(\Delta\ox \id)(E(b\ox c))\\
&\stackrel{(v)}=&
(a\ox 1\ox 1)(1\ox E)(E\ox 1)(\Delta(b)\ox c)\\
&\stackrel{(iv)}=&
(1\ox E)(a\ox 1\ox 1)(\Delta(b)\ox c)\\
&=&(\id \ox E_1)(T_2\ox \id)(a\ox b\ox c).
\end{eqnarray*}
It remains to prove the compatibility of $G_1$ with $T_2$. By Proposition
\ref{prop:F}~(1), for all $a,b,c\in A$, 
\begin{equation}\label{eq:G1_other}
G_1(a\ox bc)=(a\ox 1)((\pibarR\ox \id)T_4^\op(c\ox b)).
\end{equation}
Using this form \eqref{eq:G1_other} of $G_1$ in the first and the last
equalities and the multiplicativity of $\overline \Delta$ and Lemma
\ref{lem:coproduct} together with axiom (iv) in the second one, it follows for
any $a,b,c,d\in A$ that 
\begin{eqnarray*}
(T_2\ox \id)(\id \ox G_1)(a\ox b\ox cd)&=&
(a\ox 1\ox 1)(\Delta \ox \id)[(b\ox 1)((\pibarR\ox \id)T_4^\op(d\ox c))]\\
&=&(a\ox 1\ox 1)(\Delta(b) \ox 1)(1\ox (\pibarR\ox \id)T_4^\op(d\ox c))\\
&=&(\id \ox G_1)(T_2\ox \id)(a\ox b\ox cd),
\end{eqnarray*}
 from which we conclude by the idempotency of $A$. 

(2). Take any $a,b\in A$. By axiom (iii), $(\epsilon\ox \id)T_1(a\ox
b)=ab$. By \eqref{eq:piL}, $(\epsilon \ox \id)E_1(a\ox b)=\piL(a)b$. Finally,
using Lemma \ref{lem:counit_exp} in the second equality, it follows for all
$a,b,c\in A$ that
\begin{eqnarray*}
(\epsilon \ox \id)G_1(a\ox bc)&\stackrel{\eqref{eq:G1_other}}=&
(\id\ox \epsilon)[(1\ox a)(\id \ox \pibarR)T_4(c\ox b)]\\
&=&(\id\ox \epsilon)[(1\ox a)T_4(c\ox b)]\\
&=&(\id\ox \epsilon)[T_3(b\ox a)(c\ox 1)]\stackrel{\eqref{eq:piR(a)b}}=
\piR(a)bc.
\end{eqnarray*}

(3). $E_1^2=E_1$ is evident by the fact that $E$ is an idempotent element of
$\M(A\ox A)$. By Lemma \ref{lem:coproduct}, for any $a,b,c\in A$, 
\begin{equation}\label{eq:T4_Rmod}
(1\ox \pibarR(a))T_4(b\ox c)=
T_4(b\ox \pibarR(a) c).
\end{equation}
For any $b,c\in A$, denote $T_4(c\ox b)=:c' \ox b'$ allowing for implicit
summation. Then by Lemma \ref{lem:counital_maps}~(2),
\begin{equation}\label{eq:counital'}
\pibarR(b')c'=
\mu^\op(\id \ox \pibarR)T_4(c\ox b)=
bc.
\end{equation}
With these identities at hand, and applying Lemma \ref{lem:mod_map} in the
second equality, it follows for $a,b,c,d\in A$ that 
\begin{eqnarray*}
G_1^2(a\ox bcd)&\stackrel{\eqref{eq:G1_other}}=&
(a\pibarR(b')\ox 1)((\pibarR\ox \id)T_4^\op(d\ox c'))\\
&=&(a\ox 1)((\pibarR\ox \id)[(\pibarR(b')\ox 1)T_4^\op(d\ox c')])\\
&\stackrel{\eqref{eq:T4_Rmod}}=&
(a\ox 1)((\pibarR\ox \id)
T_4^\op(d\ox  \pibarR(b')c')\\
&\stackrel{\eqref{eq:counital'}}=&
(a\ox 1)((\pibarR\ox \id)T_4^\op(d\ox bc))
\stackrel{\eqref{eq:G1_other}}=
G_1(a\ox bcd),
\end{eqnarray*}
proving $G_1^2=G_1$. The equality $E_1T_1=T_1$ is immediate by axiom
(iv). Finally, using again the notation $T_4(c\ox b)=:c' \ox b'$ (allowing for
implicit summation), for all $a,b,c\in A$
$$
T_1G_1(a\ox bc)\stackrel{\eqref{eq:G1_other}}=
\Delta(a\pibarR(b'))(1\ox c')=
\Delta(a)(1\ox \pibarR(b')c')\stackrel{\eqref{eq:counital'}}=
\Delta(a)(1\ox bc)=
T_1(a\ox bc).
$$
The second equality follows by the multiplicativity of $\overline \Delta$,
Lemma \ref{lem:coproduct} and axiom (iv). 
\end{proof}

Applying the same reasoning as in \cite[Proposition 2.3]{VDaWa}, the following
can be shown. 

\begin{proposition}\label{prop:R1_in_L}
Let $A$ be a regular weak multiplier bialgebra over a field. If there is a
linear map $R_1:A\ox A \to A\ox A$ such that $R_1T_1=G_1$, $T_1R_1=E_1$ and
$R_1T_1R_1=R_1$, then $R_1\in \mathcal L$. 
\end{proposition}

Symmetrically to the above considerations, we can define 
$$
\begin{array}{lll}
\mathcal R:=\{K:A\ox A \to A\ox A \ |& K(ab\ox c)=(a\ox 1)K(b\ox c)\ 
\forall a,b,c\in A&\\
& (\id \ox T_1)(K \ox \id)=(K \ox \id)(\id \ox T_1)&\}.
\end{array}
$$
Denoting the vector space of left $A$-module maps $A\to A$ by
${}_A\mathsf{End}(A)$, there is an injective linear map
\begin{equation}\label{eq:R_map}
\mathcal R\to \mathsf{Lin}(A,{}_A\mathsf{End}(A)),\qquad
K\mapsto [\rho_K:a\mapsto (\id \ox \epsilon)K(-\ox a)]
\end{equation}
such that 
\begin{equation}\label{eq:R_form}
(a\ox 1)((\rho_K \ox \id)T_1(b\ox c))=K(a\ox b)(1\ox c),\qquad \forall
a,b,c\in A.
\end{equation}
For $K,K'\in \mathcal R$, 
\begin{equation}\label{eq:R_conv}
\rho_{K'K}(a)=\mu(\id \ox \rho_{K'})K(-\ox a),\quad \forall a\in A,
\end{equation}
 (where $\mu:A\ox {}_A\mathsf{End}(A)\to A$ denotes the evaluation map
$a\ox \Phi\mapsto a\Phi\equiv \Phi(a)$.) The linear map
\eqref{eq:R_map} takes the elements 
\begin{equation}\label{eq:maps_2}
T_2,\quad
E_2:a\ox b \mapsto (a\ox b)E,\quad
G_2:ab\ox c\mapsto ((\id \ox \pibarL)T_3^\op(b\ox a))(1\ox c)
\end{equation}
of $\mathcal R$ to $[a\mapsto (-)a]$, $[a\mapsto (-)\piR(a)]$ and $[a\mapsto
(-)\piL(a)]$, respectively. The equalities $E_2^2=E_2$, $G_2^2=G_2$ and
$E_2T_2=T_2=T_2G_2$ hold. If there is a linear map $R_2:A\ox A \to A\ox A$
such that $R_2T_2=G_2$, $T_2R_2=E_2$ and $R_2T_2R_2=R_2$, then $R_2\in
\mathcal R$. 

\begin{proposition}\label{prop:EG_coassoci}
Let $A$ be a regular weak multiplier bialgebra over a field and for $i\in
\{1,2\}$, let $E_i,G_i:A\ox A \to A\ox A$ be the same maps as before. Then the
following hold.
\begin{itemize}
\item[{(1)}] $(\id \ox E_1)(E_2\ox \id)=(E_2\ox \id)(\id \ox E_1)$.
\item[{(2)}] $(\id \ox G_1)(G_2\ox \id)=(G_2\ox \id)(\id \ox G_1)$.
\item[{(3)}] $(\id \ox G_1)(E_2\ox \id)=(E_2\ox \id)(\id \ox G_1)$.
\item[{(4)}] $(\id \ox E_1)(G_2\ox \id)=(G_2\ox \id)(\id \ox E_1)$.
\end{itemize}
\end{proposition}

\begin{proof}
Assertion (1) is evident and (2) follows easily by the explicit forms of
$G_1$ in \eqref{eq:G1_other} and $G_2$ in \eqref{eq:maps_2}. Concerning (3),
take any $a,b,c,d\in A$ and denote $T_4(d\ox c)=:d'\ox c'$, allowing for
implicit summation. Then using Lemma \ref{lem:coproduct} in the second 
equality, 
\begin{eqnarray*}
(\id \ox G_1)(E_2\!\!\!\!&\!\!\!\!\!\ox\!\!\!\!\!&\!\!\!\! \id)
(a\ox b\ox cd)
\stackrel{\eqref{eq:G1_other}}=
(a\ox b)E(1\ox \pibarR(c'))\ox d'\\
&=&(a\ox b\pibarR(c'))E\ox d'\stackrel{\eqref{eq:G1_other}}=
(E_2\ox \id)(\id \ox G_1)(a\ox b\ox cd).
\end{eqnarray*}
Part (4) is proven symmetrically. 
\end{proof}

Consider the vector subspace
$$
\mathcal M:= \langle 
(L,K)\in \mathcal L\times \mathcal R\ |\ 
a((\epsilon\ox \id)L(b\ox c))=
((\id \ox \epsilon)K(a\ox b))c\ \forall a,b,c\in A 
 \rangle 
$$
of $\mathcal L\times \mathcal R$. The maps \eqref{eq:L_map} and \eqref{eq:R_map}
induce a linear map 
$$
\mathcal M \to \mathsf{Lin}(A,\M(A)), \qquad 
(L,K)\mapsto [a \mapsto (\lambda_L(a),\rho_K(a))].
$$
For $(L,K)\in \mathcal M$, assume that $L=0$. Then for any $a\in A$, in
$(\lambda_L (a),\rho_K(a))\in \M(A)$ the component $\lambda_L(a)$
is zero. Hence also $\rho_K(a)=0$ for any $a\in A$ so $\rho_K=0$. Thus
by the injectivity of \eqref{eq:R_map}, also $K=0$. Symmetrically, $K=0$
implies $L=0$.

By part (2) of Proposition \ref{prop:L} and its symmetric counterpart,
$(T_1,T_2)$, $(E_1,G_2)$ and $(E_2,G_1)$ are elements of $\mathcal M$. For
$i\in \{1,2\}$, assume that there exist linear maps $R_i:A\ox A \to A\ox A$
such that $R_iT_i=G_i$, $T_iR_i=E_i$ and $R_iT_iR_i=R_i$. Then $(R_1,R_2)\in
\mathcal L\times \mathcal R$, and our next aim is to show that in fact
$(R_1,R_2)\in \mathcal M$.

\begin{proposition}\label{prop:R_coassoci}
Let $A$ be a regular weak multiplier bialgebra over a field and for $i\in
\{1,2\}$, let $E_i,G_i,R_i:A\ox A \to A\ox A$ be the same maps as before. Then
the following hold.
\begin{itemize}
\item[{(1)}] $(\id \ox E_1)(R_2\ox \id)=(R_2\ox \id)(\id \ox E_1)$.
\item[{(2)}] $(\id \ox R_1)(E_2\ox \id)=(E_2\ox \id)(\id \ox R_1)$.
\item[{(3)}] $(\id \ox G_1)(R_2\ox \id)=(R_2\ox \id)(\id \ox G_1)$.
\item[{(4)}] $(\id \ox R_1)(G_2\ox \id)=(G_2\ox \id)(\id \ox R_1)$.
\item[{(5)}] $(\id \ox R_1)(R_2\ox \id)=(R_2\ox \id)(\id \ox R_1)$.
\end{itemize}
\end{proposition}

\begin{proof}
(1). Applying part (4) of Proposition \ref{prop:EG_coassoci} in the second
 equality and its part (1) in the penultimate equality,
\begin{eqnarray*}
(\id \ox E_1)\!\!\!\!\!\!\!\!&&\!\!\!\!\!\!\!\!(R_2\ox \id)=
(\id \ox E_1)(G_2R_2\ox \id)=
(G_2\ox \id)(\id \ox E_1)(R_2\ox \id)\\
&=& (R_2T_2\ox \id)(\id \ox E_1)(R_2\ox \id)\stackrel{E_1\in \mathcal L}=
(R_2\ox \id)(\id \ox E_1)(T_2R_2\ox \id)\\
&=& (R_2\ox \id)(\id \ox E_1)(E_2\ox \id)=
(R_2E_2\ox \id)(\id \ox E_1)=
(R_2\ox \id)(\id \ox E_1).
\end{eqnarray*}
Parts (2)-(4) are proven analogously.

(5). Using part (1) of the current proposition in the second equality and its
part (3) in the penultimate equality, 
\begin{eqnarray*}
(\id \ox R_1)\!\!\!\!\!\!\!\!&&\!\!\!\!\!\!\!\!(R_2\ox \id)=
(\id \ox R_1E_1)(R_2\ox \id)=
(\id \ox R_1)(R_2\ox \id)(\id \ox E_1)\\
&=&(\id \ox R_1)(R_2\ox \id)(\id \ox T_1R_1)\stackrel{R_2\in \mathcal R}=
(\id \ox R_1T_1)(R_2\ox \id)(\id \ox R_1)\\
&=&(\id \ox G_1)(R_2\ox \id)(\id \ox R_1)=
(R_2\ox \id)(\id \ox G_1R_1)=
(R_2\ox \id)(\id \ox R_1).
\end{eqnarray*}
\end{proof}

\begin{corollary}\label{cor:S}
Let $A$ be a regular weak multiplier bialgebra over a field and for $i\in
\{1,2\}$, let $T_i,E_i,G_i:A\ox A \to A\ox A$ be the same maps as
before. Assume that there exist linear maps $R_i:A\ox A \to A\ox A$ such that
$R_iT_i=G_i$, $T_iR_i=E_i$ and $R_iT_iR_i=R_i$. Then $(R_1,R_2)\in \mathcal
M$, hence there is a corresponding linear map $S:=(\lambda_{R_1},\rho_{R_2}):A
\to \M(A)$. 
\end{corollary}

\begin{proof}
Using in the second equality that $(G_1,E_2)\in \mathcal M$, it follows for
any $a,b,c\in A$ that 
\begin{eqnarray*}
a(\epsilon \ox \id)R_1(b\ox c)&=&
a(\epsilon \ox \id)G_1R_1(b\ox c)\\
&=&\mu(\id \ox \epsilon \ox \id)(E_2\ox \id)(\id \ox R_1)(a\ox b\ox c)\\
&=&\mu(\id \ox \epsilon \ox \id)(T_2R_2\ox \id)(\id \ox R_1)(a\ox b\ox c)\\
&\stackrel{(iii)}=&
\mu(\mu \ox \id)(R_2\ox \id)(\id \ox R_1)(a\ox b\ox c).
\end{eqnarray*}
Symmetrically, using in the second equality that $(E_1,G_2)\in \mathcal M$, 
\begin{eqnarray*}
(\id \ox \epsilon)R_2(a\ox b)c&=&
(\id \ox \epsilon)G_2R_2(a\ox b)c\\
&=&\mu(\id \ox \epsilon \ox \id)(\id \ox E_1)(R_2\ox \id)(a\ox b\ox c)\\
&=&\mu(\id \ox \epsilon \ox \id)(\id \ox T_1R_1)(R_2\ox \id)(a\ox b\ox c)\\
&\stackrel{(iii)}=&
\mu(\id \ox \mu)(\id \ox R_1)(R_2\ox \id)(a\ox b\ox c).
\end{eqnarray*}
They are equal by the associativity of $\mu$ and Proposition
\ref{prop:R_coassoci}~(5). 
\end{proof}

The map $S:A\to \M(A)$ in Corollary \ref{cor:S} --- whenever it exists ---
will be termed the {\em antipode} for the following reason.

\begin{theorem}\label{thm:Hopf}
For any regular weak multiplier bialgebra $A$ over a field, there is a
bijective correspondence between the following data.
\begin{itemize}
\item[{(1)}] For $i\in \{1,2\}$, a linear map $R_i:A\ox A \to A\ox A$ such
 that $R_iT_i=G_i$, $T_iR_i=E_i$ and $R_iT_iR_i=R_i$.
\item[{(2)}] A linear map $S:A\to \M(A)$ satisfying for all $a,b,c\in A$
\begin{itemize}
\item[{(vii)}] $T_1[((\id \ox S)T_2(a\ox b))(1\ox c)]=\Delta(a)(b\ox c)$,
\item[{(viii)}] $T_2[(a\ox 1)((S\ox \id)T_1(b\ox c))]=(a\ox b)\Delta(c)$,
\item[{(ix)}] $\mu(S\ox \id)[E(a\ox 1)]=S(a)$ (equivalently, $\mu(\id\ox
 S)[(1\ox a)E]=S(a)$).
\end{itemize}
\end{itemize}
\end{theorem}

\begin{proof}
(1)$\mapsto$(2). By Corollary \ref{cor:S}, there is a linear map
$(\lambda_{R_1},\rho_{R_2})=:S:A\to \M(A)$. Using in the penultimate equality
that $T_1R_1=E_1$, 
\begin{eqnarray*}
T_1[((\id \ox S)T_2(a\ox b))(1\ox c)]&\stackrel{\eqref{eq:L_form}}=&
T_1[(a\ox 1)R_1(b\ox c)]=
\Delta(a)T_1R_1(b\ox c)\\
&=&\Delta(a)E(b\ox c)\stackrel{(iv)}=
\Delta(a)(b\ox c)
\end{eqnarray*}
for all $a,b,c\in A$, so that (vii) holds. Symmetrically, (viii) follows by
$T_2R_2=E_2$. Using in the second equality $R_1E_1=R_1$, 
$$
\mu(S\ox \id)[E(a\ox b)]\stackrel{\eqref{eq:L_conv}}=
\lambda_{R_1E_1}(a)b=
\lambda_{R_1}(a)b=
S(a)b,
$$
for all $a,b\in A$, proving the first form of (ix). The second form follows
symmetrically by $R_2E_2=R_2$. 

(2)$\mapsto$(1). By axiom (iv) in Definition \ref{def:mwba},
$$
\mathsf{Im}(T_1)\subseteq 
\langle E(a\ox b)\ |\ a,b\in A \rangle = 
\langle \Delta(a)(b\ox c)\ |\ a,b,c\in A \rangle.
$$ 
Conversely, by (vii)
$\langle \Delta(a)(b\ox c)\ |\ a,b,c\in A \rangle \subseteq 
\mathsf{Im}(T_1)$, so that $\mathsf{Im}(T_1)=\mathsf{Im}(E_1)$. 
By Proposition \ref{prop:L}~(3),
$T_1G_1=T_1$ so that $\mathsf{Ker}(G_1)\subseteq \mathsf{Ker}(T_1)$. In order
to see the converse, note that applying $\id\ox \epsilon$ to both sides of
(viii) and making use of the counitality axiom (iii) and \eqref{eq:piR(a)b}, we
conclude, since $A$ is non-degenerate, that 
\begin{equation}\label{eq:S>piR}
\mu(S\ox \id)T_1=\mu(\piR\ox \id). 
\end{equation}
Assume that for some $b,c\in A$,
$T_1(b\ox c)=0$. Then for all $a\in A$, 
\begin{eqnarray*}
0&=&
(\id \ox \mu)(\id \ox S \ox \id)(T_2\ox \id)(\id \ox T_1)(a\ox b\ox c)\\
&\stackrel{(ii)}=&
(\id \ox \mu)(\id \ox S \ox \id)(\id \ox T_1)(T_2\ox \id)(a\ox b\ox c)\\
&\stackrel{\eqref{eq:S>piR}}=&
(\id \ox \mu)(\id \ox \piR \ox \id)(T_2\ox \id)(a\ox b\ox c)=
G_1(ab\ox c)=
(a\ox 1)G_1(b\ox c).
\end{eqnarray*}
In the penultimate equality we used an alternative form of the map $G_1$ in
\eqref{eq:G1} derived from Proposition \ref{prop:F}~(1), and in the last
equality we used that $G_1$ in \eqref{eq:G1} is a morphism of left
$A$-modules. By the non-degeneracy of the multiplication in $A\ox A$, this
proves $G_1(b\ox c)=0$ hence $\mathsf{Ker}(G_1)= \mathsf{Ker}(T_1)$. By the
same reasoning applied in \cite[Proposition 2.3]{VDaWa}, the above information
about the image and the kernel of $T_1$ implies that there is a linear map
$R_1:A\ox A \to A\ox A$ with the desired properties. A bit more explicitly,
for any $a,b\in A$, 
\begin{equation}\label{eq:R<G}
R_1: T_1(a\ox b)\mapsto G_1(a\ox b), 
\end{equation}
gives $R_1$ on $\mathsf{Im}(E_1) = \mathsf{Im}(T_1)$, while $R_1$ is defined
as zero on $\mathsf{Ker}(E_1)$. The map $R_2$ is constructed
symmetrically. Note that we did not make use of property (ix) so far. 

It remains to see the bijectivity of the above correspondence. From the
expression \eqref{eq:R<G} of $R_1$, it is clear that it does not depend on the
actual choice of the map $S$ in part (2) (only on its existence). Hence
starting with the data $(R_1,R_2)$ as in part (1), we get from the
relation $R_1T_1 = G_1$ that $R_1$ must be defined by \eqref{eq:R<G} on
$\mathsf{Im}(T_1)$; and because $ R_1 = R_1 E_1$, $R_1$ must be
equal to zero on $\mathsf{Ker}(E_1)$. Similarly for $R_2$. 
Conversely, starting with a map $S$ as in part (2) and iterating the above
constructions $S\mapsto (R_1,R_2)\mapsto (\lambda_{R_1},\rho_{R_2})$, we
obtain the map $\lambda_{R_1}:A\to \mathsf{End}_A(A)$ taking 
$a\in A$ to 
\begin{eqnarray*}
b&\mapsto&
(\epsilon \ox \id)R_1(a\ox b)=
(\epsilon \ox \id)G_1R_1(a\ox b)=
\mu(\piR\ox \id)R_1(a\ox b)\\
&\stackrel{\eqref{eq:S>piR}}=&
\mu(S\ox \id)T_1R_1(a\ox b)=
\mu (S\ox \id)[E(a \ox b)].
\end{eqnarray*}
In the second equality we used Proposition \ref{prop:L}~(2). This element
$\lambda_{R_1}(a)b$ is equal to $S(a)b$ for all $a,b\in A$ if and only
if the first form of (ix) holds. Symmetrically, $a\rho_{R_2}(b)$ is equal to
$aS(b)$ for all $a,b\in A$ if and only if the second form of (ix) holds, what
proves in particular the equivalence of both stated forms of (ix). 
\end{proof}

Theorem \ref{thm:Hopf} implies in particular that if the antipode exists then
it is unique. 

Let us stress that the antipode axioms in part (2) of Theorem \ref{thm:Hopf}
imply the identities
\begin{equation}\label{eq:conv_inv}
\begin{array}{ll}
\mu(S\ox \id)T_1=\mu(\piR\ox \id),\qquad
&\mu(\id \ox S)T_2=\mu(\id \ox \piL),\\
\mu(S\ox \id)E_1=\mu(S\ \, \ox \id)\quad \Leftrightarrow
&\mu(\id \ox S)E_2=\mu(\id \ox S\ )
\end{array}
\end{equation}
expressing the requirement that $S$ is the (widely generalized) convolution
inverse of the map $A\to \M(A)$, $a\mapsto (a(-),(-)a)$. However, the
identities in \eqref{eq:conv_inv} do not seem to be equivalent to the axioms
(vii)-(ix).

Combining Theorem \ref{thm:reg_mwha>mwba} and Theorem \ref{thm:Hopf}, we
conclude that any regular weak multiplier Hopf algebra in the sense of
\cite{VDaWa} is a regular weak multiplier bialgebra in the sense of the
current paper possessing an antipode. On the other hand, if a regular weak
multiplier bialgebra in the sense of the current paper admits an antipode,
then it is also a weak multiplier Hopf algebra --- though not necessarily a
regular one --- in the sense of \cite{VDaWa}. That is to say, our regular weak
multiplier bialgebras possessing an antipode are {\em between regular and
arbitrary} weak multiplier Hopf algebras in \cite{VDaWa}. 

In view of Theorem \ref{thm:wba<>mwba}, a unital algebra possesses a weak Hopf
algebra structure as in \cite{WHAI} if and only if via the same structure
maps, it is a regular weak multiplier bialgebra with an antipode. 

From Theorem \ref{thm:Hopf} and Example \ref{ex:direct_sum}, we obtain the
following example. 

\begin{example}
For a family $\{A_j\}_{j\in I}$ of regular weak multiplier bialgebras over a
field, labelled by any index set $I$, the direct sum regular weak multiplier
bialgebra $\oplus_{j\in I} A_j$ in Example \ref{ex:direct_sum} possesses an
antipode if and only if $A_j$ does, for all $j\in I$. In this case, for any
$\underline a\in A$, $S(\underline a)= \varphi^{-1}(\{S_j(a_j)\}_{j\in I})$ in
terms of the map \eqref{eq:phi} and the antipode $S_j$ of $A_j$. 
\end{example}

Our final task is to investigate the properties of the antipode.

\begin{lemma}\label{lem:id*S}
Let $A$ be a regular weak multiplier bialgebra over a field. Assume that $A$
possesses an antipode $S:A\to \M(A)$. For $i\in \{1,2\}$, denote by $R_i:A\ox
A \to A\ox A$ the corresponding maps in Theorem \ref{thm:Hopf}~(1). Then the
following hold. 
$$
\mu R_1=\mu(\piL\ox \id)\qquad \textrm{and}\qquad
\mu R_2=\mu(\id \ox \piR).
$$
\end{lemma}

\begin{proof}
For any $a,b,c\in A$, 
$$
a (\mu R_1(b\ox c))
\stackrel{\eqref{eq:L_form}}=
(\mu(\id \ox S)T_2(a\ox b))c\stackrel{\eqref{eq:conv_inv}}=
a\piL(b)c.
$$
This proves the first assertion and the second one is proven symmetrically. 
\end{proof}

\begin{lemma}\label{lem:pi*S}
Let $A$ be a regular weak multiplier bialgebra over a field. Assume that $A$
possesses an antipode $S:A\to \M(A)$. For $i\in \{1,2\}$, denote by $R_i:A\ox
A \to A\ox A$ the corresponding maps in Theorem \ref{thm:Hopf}~(1). Then the
following hold. 
$$
\mu(\piR \ox \id)R_1=\mu(S\ox \id)\qquad \textrm{and}\qquad
\mu(\id \ox \piL)R_2=\mu(\id \ox S).
$$
\end{lemma}

\begin{proof}
Using part (2) of Proposition \ref{prop:L} in the first equality, $G_1R_1=R_1$
in the third one, and the relation between $S$ and $\lambda_{R_1}$ in Theorem
\ref{thm:Hopf} in the last one, it follows for all $a,b\in A$ that 
$$
\mu(\piR \ox \id)R_1(a\ox b)=
\mu(\lambda_{G_1} \ox \id)R_1\stackrel{\eqref{eq:L_conv}}=
\lambda_{G_1R_1}(a)b=
\lambda_{R_1}(a)b=
S(a)b.
$$
This proves the first assertion and the second one is proven symmetrically. 
\end{proof}

Although the following theorem is contained in \cite[Proposition 3.5]{VDaWa},
we prefer to give an alternative proof not referring to Sweedler type
indices. 

\begin{theorem}\label{thm:S_anti-multi}
Let $A$ be a regular weak multiplier bialgebra over a field. If $A$ possesses
an antipode $S:A\to \M(A)$, then it is anti-multiplicative. 
\end{theorem}

\begin{proof}
 For $i\in \{1,2\}$, denote by $R_i:A\ox A \to A \ox A$ the maps in
Theorem \ref{thm:Hopf}~(1). Consider the composite map
$$
W:=\mu^2(R_2\ox \id)(\id \ox \mu \ox \id)(\id \ox\id \ox R_1)
(\id \ox \mathsf{tw} \ox \id)(\id \ox\id \ox R_1)
(\id \ox \mathsf{tw} \ox \id)
$$
from $A\ox A\ox A\ox A$ to $A$. We shall evaluate it on an arbitrary element
$a\ox b\ox c \ox d$ in two different ways. In one case, we will get $aS(bc)d$
and in the other case it will yield $aS(c)S(b)d$.

To begin with, compute 
\begin{eqnarray*}
(\mu(\id \ox S)\ox \id)(\id \ox T_1)&\stackrel{\eqref{eq:R_map}}=&
(\id \ox \epsilon\ox \id)(R_2\ox \id)(\id \ox T_1)\\
&\stackrel{R_2\in \mathcal R}=&
(\id \ox \epsilon\ox \id)(\id \ox T_1)(R_2\ox \id)\\
&\stackrel{(iii)}=&
(\id \ox \mu)(R_2\ox \id).
\end{eqnarray*}
With its help, 
\begin{eqnarray*}
W(a\ox b\ox c\ox d)=
a(\mu(S\ox \id)T_1(\mu \ox \id)(\id \ox R_1)(\mathsf{tw} \ox \id)(\id \ox R_1)
(c \ox b\ox d))).
\end{eqnarray*}
Next, for all $b,c,d\in A$,
\begin{eqnarray*}
T_1(\mu \ox \id)(\id \ox R_1)(b\ox c\ox d)&=&
\Delta(b)(T_1R_1(c\ox d))
\stackrel{(iv)}=
T_1(b\ox d)(c\ox 1).
\end{eqnarray*}
Using this computation, 
\begin{eqnarray*}
W(a\ox b\ox c\ox d)&=&
a(\mu (S\ox \id)[(T_1R_1(b\ox d))(c\ox 1)])\\
&=&a(\mu(S\ox \id)E_1(bc\ox d))
\stackrel{\eqref{eq:conv_inv}}=
aS(bc)d
\end{eqnarray*}
for all $a,b,c,d\in A$.
On the other hand, by Lemma \ref{lem:id*S} and \eqref{eq:cond_exp_reg}, 
$$
\mu R_2 (\id \ox \mu)=
\mu (\id \ox \piR \mu)=
\mu (\id \ox \piR)(\id \ox \mu(\piR \ox \id))=
\mu R_2(\id \ox \mu(\piR \ox \id)),
$$
hence 
\begin{eqnarray*}
W=\mu^2 \!\!\!\!\!\!\!\!\!\!\!&&(R_2\ox \id)
(\id \ox \mu(\piR \ox \id)\ox \id)\\
&&(\id \ox\id \ox R_1)(\id \ox \mathsf{tw} \ox \id)(\id \ox\id \ox R_1)
(\id \ox \mathsf{tw} \ox \id).
\end{eqnarray*}
Moreover, for any $a,b,c,d\in A$, 
\begin{eqnarray*}
R_2(a\ox \piR(b)c)(1\ox d)&\stackrel{\eqref{eq:R_form}}=&
(a\ox 1)((S\ox \id)T_1(\piR(b)c\ox d))\\
&\stackrel{\eqref{eq:coproduct_reg}}=&
(a\ox 1)((S\ox \id)((1\ox \piR(b))T_1(c\ox d)))\\
&=&(a\ox \piR(b))((S\ox \id)T_1(c\ox d)).
\end{eqnarray*}
Therefore, 
\begin{eqnarray*}
\mu^2&&\!\!\!\!\!\!\!\!\!\!\!\!\!\!\!\!\!\!\!
(R_2\ox \id)(\id \ox \mu(\piR \ox \id)\ox \id)
(\id \ox\id \ox R_1) (a\ox b\ox c\ox d)\\
&=&\mu[(a\ox \piR(b))((S\ox \id)T_1R_1(c\ox d))]\\
&=&\mu[(a\ox \piR(b))((S\ox \id)(E(c\ox d)))]\\
&\stackrel{\eqref{eq:coproduct_reg}}=&
a(\mu(S\ox \id)E_1(c\ox \piR(b)d))
\stackrel{\eqref{eq:conv_inv}}=
aS(c)\piR(b)d.
\end{eqnarray*}
Substituting this identity and applying Lemma \ref{lem:pi*S}, we obtain
$$
W(a\ox b\ox c\ox d)=
\mu^3(\id \ox S \ox \piR\ox \id)(\id \ox\id \ox R_1) (a\ox c\ox b\ox d)=
aS(c)S(b)d
$$
for any $a,b,c,d\in A$. By the density of $A$ in $\M(A)$, this proves
$S(bc)=S(c)S(b)$, for all $b,c\in A$.
\end{proof}

The following proposition is contained in \cite[Proposition
3.6]{VDaWa}. However, in our setting a much shorter proof can be given. 

\begin{proposition}\label{prop:S_nd}
Let $A$ be a regular weak multiplier bialgebra over a field which 
possesses an antipode $S:A\to \M(A)$. Whenever the comultiplication is left
and right full, $S$ is a non-degenerate map. 
\end{proposition}

\begin{proof}
Using the idempotency of the algebra $A$, Lemma \ref{lem:counital_maps}~(1),
the fact that $\piL(A)=\pibarL(A)$ (cf. Theorem \ref{thm:full}) and
\eqref{eq:conv_inv}, 
$$
A=A^2\subseteq A\pibarL(A) = A\piL(A)\subseteq AS(A)\subseteq A
$$
so that $A=AS(A)$. A symmetrical reasoning shows that also $A=S(A)A$.
\end{proof}

We conclude by Theorem \ref{thm:extend} that in the situation in Proposition
\ref{prop:S_nd} the antipode extends to algebra homomorphisms $\overline
S:\M(A)^\op\to \M(A)$, $\overline {\id\ox S}:\M(A\ox A^\op)\cong\M(A^\op\ox
A)^\op\to \M(A\ox A)$, $\overline {S \ox \id}:\M(A\ox A)\to \M(A^\op\ox
A)\cong \M(A\ox A^\op)^\op$ and $\overline {S \ox S}: \M(A\ox A)^\op\to
\M(A\ox A)$. 

\begin{lemma}\label{lem:S_mod_map}
Let $A$ be a regular weak multiplier bialgebra over a field. Assume that $A$
possesses an antipode $S:A\to \M(A)$. Then for any $a,b\in A$, the following
hold. 
\begin{eqnarray*}
&S(a\pibarL(b))=\piR(b)S(a)\qquad 
&S(\pibarL(b)a)=S(a)\piR(b)\\
&S(a\pibarR(b))=\piL(b)S(a)\qquad 
&S(\pibarR(b)a)=S(a)\piL(b).
\end{eqnarray*}
\end{lemma}

\begin{proof}
Using Lemma \ref{lem:E_balanced} in the second equality, it follows for any
$a,b,c\in A$ that 
\begin{eqnarray*}
a S(b\pibarL(c))&\stackrel{\eqref{eq:conv_inv}}=&
\mu(\id \ox S)[(a\ox b\ \pibarL(c))E]\\
&=&\mu(\id \ox S)[(a\piR(c)\ox b)E]\stackrel{\eqref{eq:conv_inv}}=
a \piR(c) S(b).
\end{eqnarray*}
 By the density of $A$ in $\M(A)$, this proves the first claim. It
also implies that 
$$
S(a)\piR(c) S(b)d=S(a)S(b\pibarL(c))d=S(\pibarL(c)a)S(b)d
$$
for all $a,b,c,d\in A$, where in the second equality we used the
anti-multiplicativity of $S$. Using the non-degeneracy of $S$ and the density
of $A$ in $\M(A)$, we have the second claim proven. The remaining
assertions follow symmetrically.
\end{proof}

In view of Proposition \ref{prop:(a@1)E} and \eqref{eq:E(a@1)}, in any
regular weak multiplier bialgebra $A$ over a field, we may regard $E$ as an
element of $\M(A\ox A^\op)$. The following proposition --- and thus its
Corollary \ref{cor:S_anti-comulti} --- was proven in \cite{VDaWa} only for
regular weak multiplier Hopf algebras. 

\begin{proposition}\label{prop:E<S>F}
Let $A$ be a regular weak multiplier bialgebra over a field with a left
and right full comultiplication. Assume that $A$ possesses an antipode $S:A\to
\M(A)$. Then the elements $E\in \M(A\ox A^\op)$ and $F\in \M(A\ox A)$ in
Proposition \ref{prop:F}~(1) are related via the extensions of $S$ as
$$
(\overline {\id\ox S})(E)=F \qquad \textrm{and} \qquad 
(\overline {S \ox \id})(F)=E^\op,
$$
where $(a\ox 1)E^\op(1\ox b):=\mathsf{tw}[(1\ox a)E(b\ox 1)]$ and 
$(1\ox b)E^\op(a\ox 1):=\mathsf{tw}[(b\ox 1)E(1\ox a)]$ define $E^\op\in
\M(A^\op\ox A)$.
\end{proposition}

\begin{proof}
Since $A$ is idempotent and $S$ is non-degenerate (by Proposition
\ref{prop:S_nd}), any element of $A\ox A$ can be written as a linear
combination of elements of the form $ab\ox cS(d)$, in terms of $a,b,c,d\in
A$. Moreover, using the anti-multiplicativity of $S$ in the first equality,
applying Proposition \ref{prop:(a@1)E}~(1) in the second equality, Lemma
\ref{lem:S_mod_map} in the third one and Proposition \ref{prop:F}~(1) in the
last one, it follows for any $a,b,c,d\in A$ that 
\begin{eqnarray*}
(ab\ox cS(d))(\overline {\id\ox S})(E)&=&
(1\ox c)((\id \ox S)[(ab\ox 1)E(1\ox d)])\\
&=&(1\ox c)((\id \ox S)[((\id \ox \pibarL)T_2(a\ox b))(1\ox d)]\\
&=&(1\ox cS(d))((\id \ox \piR)T_2(a\ox b))=
(ab\ox cS(d))F.
\end{eqnarray*}
This proves the first assertion. Symmetrically, in order to prove the second
one, write any element of $A\ox A$ as a linear combination of elements of the
form $aS(b)\ox cd$. Using again the anti-multiplicativity of $S$ in the first
equality, Proposition \ref{prop:F}~(1) in the second equality, Lemma
\ref{lem:S_mod_map} in the third one and \eqref{eq:E(a@1)} in the fourth
one, it follows for any $a,b,c,d\in A$ that 
\begin{eqnarray*}
(aS(b)\ox 1)(\overline {S\ox \id})(F)(1\ox cd)&=&
(a\ox 1)((S \ox \id)[F(b\ox cd)])\\
&=&(a\ox 1)((S \ox \id)[((\pibarR \ox \id)T_4^\op(d\ox c))(b\ox 1)]\\
&=&(aS(b)\ox 1)) ((\piL \ox \id)T_4^\op(d\ox c))\\
&=&(aS(b)\ox 1)\mathsf{tw}(E(cd\ox 1))\\
&=&\mathsf{tw}((1\ox aS(b))E(cd\ox 1))\\
&=&(aS(b)\ox 1)E^\op(1\ox cd).
\end{eqnarray*}
\end{proof}

\begin{corollary}\label{cor:S_anti-comulti}
Let $A$ be a regular weak multiplier bialgebra over a field with a left
and right full comultiplication. If $A$ possesses an antipode $S:A\to
\M(A)$, then $S$ is anti-comultiplicative in the sense of the commutative
diagram 
$$
\xymatrix@C=60pt{
A^\op\ar[r]^-S\ar[d]_-{\Delta^\op}&
\M(A)\ar[d]^-{\overline\Delta}\\
\M(A\ox A)^\op\ar[r]_-{\overline{S\ox S}}&
\M(A\ox A).}
$$
\end{corollary}

\begin{proof}
By \cite[Proposition 3.7]{VDaWa}, $\overline \Delta S(a)=((\overline{S\ox
S})\Delta^\op(a))E$, for all $a\in A$. By Proposition \ref{prop:E<S>F},
$E=(\overline{S\ox S})(E^\op)$ so that by the anti-multiplicativity of
$\overline{S\ox S}$, 
\begin{eqnarray*}
\overline \Delta S(a)&=&
((\overline{S\ox S})\Delta^\op(a))E=
((\overline{S\ox S})\Delta^\op(a))(\overline{S\ox S})(E^\op)\\
&=&(\overline{S\ox S})(E^\op\Delta^\op(a))\stackrel{(iv)}=
(\overline{S\ox S})\Delta^\op(a).
\end{eqnarray*}
\end{proof}

\end{document}